\documentclass[10pt]{article}
\usepackage{amssymb}
\usepackage{color}
\usepackage{relsize}
\usepackage{booktabs}
\usepackage{amsthm}
\usepackage{amsmath}
\usepackage{amsfonts}
\usepackage{enumerate}

\usepackage{graphicx}
\usepackage{float}
\usepackage{subfigure}
\usepackage{soul}
\usepackage{url}
\usepackage{color}

\newtheorem{thm}{Theorem}[section]
\newtheorem{cor}[thm]{Corollary}
\newtheorem{lem}[thm]{Lemma}
\newtheorem{prop}[thm]{Proposition}

\theoremstyle{definition}
\newtheorem{defn}[thm]{Definition}
\newtheorem{rem}[thm]{Remark}
\newtheorem{assu}[thm]{Assumption}

\numberwithin{equation}{section}

\theoremstyle{definition}

\theoremstyle{remark}

\usepackage{fancyhdr}
\makeatletter
\let\runauthor\@author
\let\runtitle\@title
\makeatother
\begin{document}
	\title{Solving Parametric Partial Differential Equations with Deep Rectified Quadratic Unit Neural Networks}
	\author{Zhen Lei \footnotemark[1] \footnotemark[2] \and Lei Shi \footnotemark[1] \footnotemark[3] \and Chenyu Zeng \footnotemark[1] \footnotemark[4]
	}
	\renewcommand{\thefootnote}{\fnsymbol{footnote}}
	\footnotetext[1]{School of Mathematical Sciences; LMNS and Shanghai
		Key Laboratory for Contemporary Applied Mathematics, Fudan University, Shanghai 200433, P. R. China.}
	\footnotetext[2]{Email: zlei@fudan.edu.cn}
	\footnotetext[3]{Email: leishi@fudan.edu.cn}
	\footnotetext[4]{Correspondence author. Email: cyzeng19@fudan.edu.cn}
	\date{\today}
	\maketitle
	
	\begin{abstract}
		Implementing deep neural networks for learning the solution maps of parametric partial differential equations (PDEs) turns out to be more efficient than using many conventional numerical methods. However, limited theoretical analyses have been conducted on this approach. In this study, we investigate the expressive power of deep rectified quadratic unit (ReQU) neural networks for approximating the solution maps of parametric PDEs. The proposed approach is motivated by the recent important work of G. Kutyniok, P. Petersen, M. Raslan and R. Schneider (Gitta Kutyniok, Philipp Petersen, Mones Raslan, and Reinhold Schneider. A theoretical analysis of deep neural networks and parametric pdes. Constructive Approximation, pages 1-53, 2021), which uses deep rectified linear unit (ReLU) neural networks for solving parametric PDEs. In contrast to the previously established complexity-bound $\mathcal{O}\left(d^3\log_{2}^{q}(1/ \epsilon) \right)$ for ReLU neural networks, we derive an upper bound $\mathcal{O}\left(d^3\log_{2}^{q}\log_{2}(1/ \epsilon) \right)$ on the size of the deep ReQU neural network required to achieve accuracy $\epsilon>0$, where $d$ is the dimension of reduced basis representing the solutions. Our method takes full advantage of the inherent low-dimensionality of the solution manifolds and better approximation performance of deep ReQU neural networks. Numerical experiments are performed to verify our theoretical result.
	\end{abstract}
	

	\maketitle
	



	\section{Introduction}	
	Solving partial differential equations (PDEs) numerically has attracted considerable research attention because of its potential application in many science and engineering problems. Conventional numerical methods, such as finite element and finite difference methods, are primarily used for solving low-dimensional PDEs. In these methods, an equation is solved by discretizing the domain, which could be sufficiently accurate with fine grids. However, conventional methods are often limited by the curse of dimensionality, as they approximate the function values of the solution on grid points which scale exponentially with the input dimension. On the other hand, in most applications, the concerned PDE typically depends on several parameters that describe the physical or geometrical constraints of the equation. Under certain circumstances, PDEs are solved repeatedly for various values of these parameters. Thus, solving PDEs using conventional methods could be time consuming.
	
	Because of the expressive power of neural networks, considerable advancements have been achieved in using deep learning methods for solving PDEs. A flourishing line of research leverages neural networks to approximate the solutions of PDEs directly (see, e.g. \cite{han2018solving,sirignano2018dgm}), producing efficient numerical methods which outperform the conventional methods in solving high-dimensional PDEs. However, this line of work still aims to build one solution from a parameter set to a state space and thus cannot avoid the trouble of retraining for different parameters. To address this problem, studies (\cite{bhattacharya2021model,DALSANTO2020109550,khoo_lu_ying_2021, li2020fourier, lu2019deeponet}) have proposed the use of neural networks for recovering the solution maps. This approach can be used to solve an entire family of PDEs. However, the theoretical foundations, including the required network size, for this approach have not yet been comprehensively investigated.
	
	In the recent important work \cite{kutyniok2021theoretical}, the authors have made a significant  breakthrough, showing that an efficient approximation of the solution map can be realized using deep neural networks and the low dimensionality of the solution manifold. The following parametric PDEs have been considered:
	\begin{equation}\label{PPDE}
		B_{y}(u_y,v)=f_y(v) \quad \text{~for all~} y\in \mathcal{Y}, v\in H,
	\end{equation}
	where $B_{y}:H\times H\to \mathbb{R}$ is a continuous bilinear form, $\mathcal{Y}$ is the set of parameters, and $H$ is a Hilbert space. The authors of \cite{kutyniok2021theoretical} focused on approximating the solution map $\mathcal{Y}\ni y \mapsto  u_y \in H$. Based on the Galerkin method and the compactness of parameter set $\mathcal{Y}$, one can assume that there exists a basis of a high-fidelity discretization of $H$ which may potentially be quite large and the corresponding  high-fidelity solutions $u^{\mathrm{h}}_y$. Combined with the reduced basis theory, the high-fidelity space could be replaced by a considerably smaller space of dimension $d$, that is, the reduced basis space. And it suffices to approximate the reduced  basis solution $u^{\mathrm{rb}}_y$. In terms of the coefficient vector, this approach transformed the original problem into solving a linear system of equations, which are presented in Section \ref{RBS}. The main contribution of \cite{kutyniok2021theoretical} is that they constructed a neural network equipped with the rectified linear unit (ReLU) activation function of size $\mathcal{O}(d^3\log^q_2(1/\epsilon))$ to approximate the matrix inversion operator up to accuracy $\epsilon>0$, where $d$ is the dimension of reduced basis representing the solutions and the ReLU activation function is defined as $x \to \max\{0,x\}$. Furthermore, they obtained the upper bounds of the complexity of using deep ReLU neural networks for approximating the solution maps of parametric PDEs. Since then, it remains a natural question whether this complexity bound could be further improved when some other neural network architectures are adopted.
	
	In this study, we propose the use of deep neural networks equipped with the rectified quadratic unit (ReQU) activation function to approximate the solution maps of parametric PDEs. Here, ReQU is defined as $x \to \max\{0,x^2\}.$ Compared with the approximation by deep ReLU neural networks, we obtain a better complexity estimate of order $\mathcal{O}\left(d^3\log_{2}^{q}\log_{2}(1/ \epsilon) \right)$ when solving the parametric PDEs with $\epsilon$-accuracy using deep ReQU neural networks. 
	
	The ReQU is considered as the activation function because as indicated in \cite{li2019better}, ReQU neural networks exhibit superior approximation of smooth functions, such as polynomials. This is achieved by the fact that functions $x, x^2$, and $xy$ can be exactly represented using ReQU neural networks with less nodes and nonzero weights, and without restriction on the norm of $x$ and $y$. Based on these properties, we propose using ReQU neural networks to represent matrix polynomials, which contributes to an approximation of the inversion operator. A key difficulty in implementing our strategy is deriving a ReQU-based neural network calculus. We overcome this problem using a novel construction of identity function and complexity estimates of operations of neural networks equipped with the ReQU activation function.
	
	Furthermore, another important ingredient of our study is the inherent low dimensionality of the solution manifold. This forms the foundation of reduced basis method \cite{hesthaven2016certified,kutyniok2021theoretical,quarteroni2015reduced}, where the parameter-dependent solutions are obtained in a low-dimensional space constructed using the snapshots of high-dimensional discrete solutions for selected parameters. Typically, this reduced space is built up in advance in a costly offline phase, and is suited to approximate all solutions for parameters stemming from a given parameter domain during the so-called online phase. We refer to \cite{quarteroni2015reduced} for an extensive survey of works on reduced basis methods.
	
	Numerous theoretical analysis based on neural networks have been proposed for solving parametric PDEs. The work \cite{schwab2019deep} was one of the first to present a method for analyzing neural network approximation rates for solutions of parametric PDEs. In their approach, the analyticity of the solution map and polynomial chaos expansion were used to approximate the solution map. In \cite{kovachki2021universal,lanthaler2021error}, detailed error estimates of neural networks were obtained for Fourier neural operators and DeepONets. We also mention that ReQU neural networks have already been used in \cite{chui2018deep, mhaskar1996neural, mhaskar1993approximation}) to provide a theoretical upper bound on deep neural network function approximations, which was improved by \cite{li2019better,tang2019chebnet} for smooth functions and in \cite{opschoor2021exponential} for holomorphic maps in high dimension. The ReQU neural network has been adapted to solve PDEs in \cite{duan2021convergence, he2018relu}. To the best of our knowledge, this study is the first to use ReQU neural networks for solving parametric PDEs.
	
	The rest of this paper is organized as follows. In Section 2, we present the framework of the parametric PDEs problem, and briefly describe the theory of reduced basis method and rectified power units (RePU) activation function. In Section 3, we introduce the ReQU-based neural network calculus, which is the basis for assembling complex neural networks. In Section 4, we construct the ReQU neural network for mapping a matrix to its approximate inverse, and prove its upper bound of complexity. We describe the final approximate solution to parametric PDEs constructed by ReQU neural networks and the corresponding complexity estimates in Section 5. In Section 6, we detail some numerical experiments to support the theoretical analysis results.
	
	\section{Preliminary}\label{section2}
	\subsection{Well-Posedness Theory of Parametric PDEs}
	
	This subsection introduces the assumptions that guarantee the well-posedness of parametric PDEs. The parametric PDEs are given in \eqref{PPDE} with a continuous bilinear form $B_{y}:H\times H\to \mathbb{R}$ and a set of parameters denoted by $\mathcal{Y}.$ Here, $H$ is a Hilbert space with norm $\|\cdot\|_{H}$, and $\mathcal{Y}$ is a compact subset of $\mathbb{R}^{p}$ with a fixed, but potentially large,  $p\in \mathbb{N}$.
	\begin{assu}\label{assu}
		The bilinear form $B_{y}$ satisfies the following conditions: 
		\begin{enumerate}[(i)]
			\item $
			B_{y}(u, v)=B_{y}(v, u), \quad \text { for all } u, v \in H, y\in \mathcal{Y}.$
			\item There exists a constant $\alpha>0$ such that $$
			\left|B_{y}(u, v)\right| \leq \alpha\|u\|_{H}\|v\|_{H}, \quad \text { for all } u \in H, v \in H, y \in \mathcal{Y}.
			$$
			\item  There exists a constant $\beta>0$ such that
			$$
			B_{y}(u, u) \geq \beta \|u\|_{H}^{2}, \quad \text { for all } u \in H, y \in \mathcal{Y}.
			$$
		\end{enumerate} 
		Moreover, we assume that there exists a constant $C_f>0$ such that
		$$
		\left\|f_{y}\right\|_{H} \leq C_f, \quad \text { for all } y \in \mathcal{Y},$$
		and the solution manifold 
		$$
		S(\mathcal{Y}):=\left\{u_{y}: u_{y} \text { is the solution of }\eqref{PPDE}, y \in \mathcal{Y}\right\}
		$$
		is compact in $H$. 
	\end{assu}
	
	Next, based on the Lax--Milgram lemma \cite{MR2597943}, the parametric PDE \eqref{PPDE} is well-posed; that is, for every $y \in \mathcal{Y}$ and every $f_{y} \in H$, there exists exactly one $u_{y} \in H$ such that \eqref{PPDE} is satisfied, and $u_{y}$ depends continuously on $f_{y}$.
	
	Parametric PDEs of form \eqref{PPDE} are widely used throughout engineering and applied sciences as models for unsteady and steady heat and mass transfer, acoustics, solid and fluid mechanics, electromagnetics, or finance problems \cite{hesthaven2016certified}. The parameters in $\mathcal{Y}$ are used to characterize a particular equation and possible variations in its geometric configuration, physical properties, boundary conditions, or source terms. 
	
	Typical examples of parametric PDEs include the boundary-value problem of linear parametric elliptic PDE with variable coefficients
	\begin{align}\label{BV}
		\begin{cases}
			-\sum_{i, j=1}^{n}\left(a_y^{i j}(x) u_{x_{i}}\right)_{x_{j}}+\sum_{i=1}^{n} b_y^{i}(x) u_{x_{i}}+c_y(x) u =f_y, & \text{~in~} U,\\
			u=0 & \text{~on~} \partial U,
		\end{cases}
	\end{align}
	where $U$ is an open bounded subset of $\mathbb{R}^n$.  We assume that, for every fixed $y\in \mathcal{Y}$, there exists a constant $\theta_y>0$ such that
	$$
	\sum_{i, j=1}^{n} a_y^{i j}(x) \xi_{i} \xi_{j} \geq \theta_y|\xi|^{2}
	$$
	for a.e. $x \in U$ and all $\xi \in \mathbb{R}^{n}$. 
	The bilinear form $B_y$ associated with \eqref{BV} is expressed as follows:
	\begin{align}\label{eq-exm-B}
		B_y(u, v):=\int_{U} \sum_{i, j=1}^{n} a_y^{i j} u_{x_{i}} v_{x_{j}}+\sum_{i=1}^{n} b_y^{i} u_{x_{i}} v+c_y u v d x
	\end{align}
	for $u, v \in H_{0}^{1}(U)$, which satisfies conditions (i)-(iii) in Assumption \ref{assu}. 

	\subsection{Reduced Basis Solution}\label{RBS}
	In practice, we cannot solve the equation \eqref{PPDE} exactly for every $y\in \mathcal{Y}$. Instead, we construct a high-fidelity solution $u_y^{\mathrm{h}}$ for every fixed $y$, which is achieved using the standard Galerkin method, that is, solving \eqref{PPDE} on a subspace $H_{m} \subset H$ of dimension $m$ such that
	\begin{equation}\label{Galerkin-sol}
		B_{y}\left(u_{y}^{m}, e_{i}\right)=f_{y}(e_{i}) \quad \text { for } i=1,2,\dots,m,
	\end{equation}
	where $\{e_{1}, e_{2}, \ldots, e_{m}\}$ is a basis of $H_{m}$, and $u_{y}^{m}$ can be represented as
	$u_{y}^{m}=\sum_{j=1}^{m} \left(\mathbf{u}^{m}_{y}\right)_{j} e_{j}.$
	We denote $\mathbf{B}^{m}_{y}:=\left(B_{y}\left(e_{j}, e_{i}\right)\right)_{i, j=1}^{m}\in \mathbb{R}^{m\times m}$ and $\mathbf{f}^{m}_{y}:=\left(f_{y}\left(e_{i}\right)\right)_{i=1}^{m} \in \mathbb{R}^{m}$, then equation \eqref{Galerkin-sol} is equivalent to the following system of linear equations:
	\begin{equation*}
		\mathbf{B}^{m}_{y} \mathbf{u}^{m}_{y} = \mathbf{f}^{m}_{y}, \quad \text{i.e.~} \ 
		\mathbf{u}^{m}_{y}=\left(\mathbf{B}^{m}_{y}\right)^{-1} \mathbf{f}^{m}_{y} \in \mathbb{R}^{m}.
	\end{equation*}
	Furthermore, based on C\'{e}a's lemma \cite{quarteroni2015reduced}, the error $u_y-u_{y}^{m}$ between the original and Galerkin solution admits the estimate
	\begin{equation}\label{cea}
		\left\|u_{y}-u_{y}^m\right\|_{H} \leq \frac{\alpha}{\beta} \inf _{w \in H_m}\left\|u_{y}-w\right\|_{H}.
	\end{equation}
	Thus, because of the compactness of the parametric set $\mathcal{Y}$, we assume that, for an arbitrarily small, but fixed $\hat{\epsilon}>0 $, there exists a high-fidelity space $H^h \subset H$ with dimension $D<$ $\infty$ and basis $\left(e_{i}\right)_{i=1}^{D},$ such that the corresponding high-fidelity solutions $u_{y}^{\mathrm{h}}$ satisfy
	$$ \sup _{y \in \mathcal{Y}}\left\|u_{y}-u_{y}^{\mathrm{h}}\right\|_{H} \leq \hat{\epsilon}.$$
	We denote $\mathbf{B}_{y}^{\mathrm{h}}:=$ $\left(B_{y}\left(e_{j}, e_{i}\right)\right)_{i, j=1}^{D} \in \mathbb{R}^{D \times D}$, $\mathbf{f}_{y}^{\mathrm{h}}:=\left(f_{y}\left(e_{i}\right)\right)_{i=1}^{D}$, and $\mathbf{u}_{y}^{\mathrm{h}}:=\left(\mathbf{B}_{y}^{\mathrm{h}}\right)^{-1} \mathbf{f}_{y}^{\mathrm{h}} \in \mathbb{R}^{D}$, then we have $$u_{y}^{\mathrm{h}}:=\sum_{i=1}^{D}\left(\mathbf{u}_{y}^{\mathrm{h}}\right)_{i} e_{i}.$$
	
	However, in practice, solving the high-fidelity solution for various parameters entails severe computational costs. Following \cite{kutyniok2021theoretical}, instead of the high-fidelity solution $u_y^{\mathrm{h}}$, we aim to learn the reduced basis solution $u_y^{\mathrm{rb}}$ of parametric PDE. The reduced basis solution is in a considerably lower dimension space than the high-fidelity solution, while the error between the original solution are comparable. 
	Specifically, we assume that, for every $\tilde{\epsilon} \geq \hat{\epsilon}$, there exists a reduced basis space  $H_{\tilde{\epsilon}}^{\mathrm{rb}}\subset H$, such that $d(\tilde{\epsilon}):=\operatorname{dim}\left(H_{\tilde{\epsilon}}^{\mathrm{rb}}\right) \ll D$ and 
	\begin{equation}\label{fe}
		\sup _{y \in \mathcal{Y}} \inf _{w \in H_{\tilde{\epsilon}}^{\mathrm{rb}}}\left\|u_{y}-w\right\|_{H} \leq \tilde{\epsilon}.
	\end{equation}
	The existence of the reduced basis is based on the compactness of the solution manifold $S(\mathcal{Y})$. A detailed discussion is presented in \cite{quarteroni2015reduced}. 
	
	A basis of $H_{\tilde{\epsilon}}^{\mathrm{rb}}$ is constructed from a set of suitable high-fidelity solutions $\left\{u^{\mathrm{h}}_{y^{1}}, \ldots, u^{\mathrm{h}}_{y^{N}}\right\}$ that we call snapshots, corresponding to a set of $N$ selected parameters
	$
	\mathcal{Y}_{N}=\left\{y^{1}, \ldots, y^{N}\right\} \subset \mathcal{Y},
	$
	such that
	\begin{equation*}
		\sup _{y \in \mathcal{Y}} \inf _{w \in \operatorname{span}\left\{u^{\mathrm{h}}_{y^{1}}, \ldots, u^{\mathrm{h}}_{y^{N}}\right\} }\left\|u_{y}-w\right\|_{H} \leq \tilde{\epsilon}.
	\end{equation*}
	Next, by orthonormalizing the snapshots, we can generate a set of $N$ functions $
	\left\{\psi_{1}, \ldots, \psi_{N}\right\}$ called the reduced basis. Through construction, the reduced basis functions are no longer solutions to the high-fidelity problem, but we have $$ H_{\tilde{\epsilon}}^{\mathrm{rb}} :=\operatorname{span}\left\{\psi_{1}, \ldots, \psi_{N}\right\}=\operatorname{span}\left\{u^{\mathrm{h}}_{y^{1}}, \ldots, u^{\mathrm{h}}_{y^{N}}\right\} $$
	and $d(\tilde{\epsilon}):=\operatorname{dim}\left(H_{\tilde{\epsilon}}^{\mathrm{rb}}\right) =N \ll D.$ The basis vectors $\left(\psi_{i}\right)_{i=1}^{d(\tilde{\epsilon})}$ are linear combinations of the high-fidelity basis vectors $\left(e_{i}\right)_{i=1}^{D}$ with a transformation matrix $\mathbf{V}_{\tilde{\epsilon}} \in \mathbb{R}^{D \times d(\tilde{\epsilon})}$, such that 
	$$
	\left(\psi_{i}\right)_{i=1}^{d(\tilde{\epsilon})}=\left(\sum_{j=1}^{D}\left(\mathbf{V}_{\tilde{\epsilon}}\right)_{j, i} e_{j}\right)_{i=1}^{d(\tilde{\epsilon})}.
	$$
	The reduced basis solution can be represented as follows:
	\begin{equation}\label{represent}
		u_{y, \tilde{\epsilon}}^{\mathrm{rb}}=\sum_{i=1}^{d(\tilde{\epsilon})}\left(\mathbf{u}_{y, \tilde{\epsilon}}^{\mathrm{rb}}\right)_{i} \psi_{i}=\sum_{j=1}^{D}\left(\mathbf{V}_{\tilde{\epsilon}} \mathbf{u}_{y, \tilde{\epsilon}}^{\mathrm{rb}}\right)_{j} e_{j}.
	\end{equation}
	We denote
	$
	\tilde{\mathbf{u}}_{y, \tilde{\epsilon}}^{\mathrm{h}}:=\mathbf{V}_{\tilde{\epsilon}} \mathbf{u}_{y, \tilde{\epsilon}}^{\mathrm{rb}} \in \mathbb{R}^{D}
	$
	the coefficient vector of the reduced basis solution if it is expanded with respect to the high-fidelity basis $\left(e_{i}\right)_{i=1}^{D} .$
	
	Given $y \in \mathcal{Y}$, the reduced basis approximation of equation \eqref{PPDE} is written as follows:
	\begin{equation}\label{rba}
		B_{y}\left(u_{y, \tilde{\epsilon}}^{\mathrm{rb}}, v\right)=f_{\mathrm{y}}(v) \quad \text { for all } v \in H_{\tilde{\epsilon}}^{\mathrm{rb}}.
	\end{equation}
	By inserting \eqref{represent} into \eqref{rba}, we have (see for example \cite{quarteroni2015reduced} Section 3.4.1) 
	\begin{equation}\label{rb}
		\mathbf{u}_{y, \tilde{\epsilon}}^{\mathrm{rb}}:=\left(\mathbf{B}_{y, \tilde{\epsilon}}^{\mathrm{rb}}\right)^{-1} \mathbf{f}_{y, \tilde{\epsilon}}^{\mathrm{rb}},
	\end{equation}
	where
	\begin{equation*}
		\mathbf{B}_{y, \tilde{\epsilon}}^{\mathrm{rb}}:=\left(B_{y}\left(\psi_{j}, \psi_{i}\right)\right)_{i, j=1}^{d(\tilde{\epsilon})}=\mathbf{V}_{\tilde{\epsilon}}^{T} \mathbf{B}_{y}^{\mathrm{h}} \mathbf{V}_{\tilde{\epsilon}} \in \mathbb{R}^{d(\tilde{\epsilon}) \times d(\tilde{\epsilon})}, \quad \text { for all } y \in \mathcal{Y}
	\end{equation*}
	and
	\begin{equation*}
		\mathbf{f}_{y, \tilde{\epsilon}}^{\mathrm{rb}}:=\left(f_{y}\left(\psi_{i}\right)\right)_{i=1}^{d(\tilde{\varepsilon})}=\mathbf{V}_{\tilde{\epsilon}}^{T} \mathbf{f}_{y, \tilde{\epsilon}}^{\mathrm{h}} \in \mathbb{R}^{d(\tilde{\epsilon})}.
	\end{equation*}
	Then, by using \eqref{cea}, we derive the following expression:
	\begin{equation*}
		\sup _{y \in \mathcal{Y}}\left\|u_{y}-u_{y, \tilde{\epsilon}}^{\mathrm{rb}}\right\|_{H} \leq  \frac{\alpha}{\beta}  \sup _{y \in \mathcal{Y}}\inf _{w \in H_{\tilde{\epsilon}}^{\mathrm{rb}}}\left\|u_{y}-w\right\|_{H} \leq \frac{\alpha}{\beta} \tilde{\epsilon}.
	\end{equation*}
	Moreover, because bilinear form $B_y(\cdot,\cdot)$ is symmetric and
	coercive and the reduced basis functions are orthonormal, we have the following expression:
	\begin{equation}\label{B}
		\beta \leq\left\|\mathbf{B}_{y, \tilde{\epsilon}}^{\mathrm{rb}}\right\|_{2} \leq \alpha, \quad \text { as well as } \quad \alpha^{-1} \leq\left\|\left(\mathbf{B}_{y, \tilde{\epsilon}}^{\mathrm{rb}}\right)^{-1}\right\|_{2} \leq \beta^{-1}, \quad \text { for all } y \in \mathcal{Y},
	\end{equation}
	and 
	$$\left|\mathbf{f}_{y, \tilde{\epsilon}}^{\mathrm{rb}}\right| \leq\left\|f_{y}\right\|_{H}  \leq C_{f},$$
	where $|\cdot|$ denotes the Euclidean norm, and $\|\cdot\|_{2}$ denotes the spectral norm.
	
	Finally, let $\mathbf{G}:=\left(\left\langle e_{i},e_{j}\right\rangle_{H}\right)_{i, j=1}^{D} \in \mathbb{R}^{D \times D}$ be the symmetric, positive definite Gram matrix of the high-fidelity basis $\left(e_{i}\right)_{i=1}^{D}$. Next, for any $v \in H_h$ with coefficient vector $\mathbf{v}$ with respect to the basis $\left(e_{i}\right)_{i=1}^{D}$, we have the following expression:
	\begin{equation*}
		|\mathbf{v}|_{\mathbf{G}}:=\left|\mathbf{G}^{1 / 2} \mathbf{v}\right|=\|v\|_{H}.
	\end{equation*}
	Because the reduced basis functions are orthonormal, we have the following expression:
	\begin{equation}\label{G1/2}
		\left\|\mathbf{G}^{1 / 2} \mathbf{V}_{\tilde{\epsilon}}\right\|_{2}=1, \quad \text { for all } \tilde{\epsilon} \geq \hat{\epsilon},  
	\end{equation}
	and 
	$$\left\|\sum_{i=1}^{d(\tilde{\epsilon})} \mathbf{c}_{i} \psi_{i}\right\|_{H}=|\mathbf{c}|, \quad \text { for all } \mathbf{c} \in \mathbb{R}^{d(\tilde{\epsilon})}.$$
	\subsection{Rectified Power Units}
	In this subsection, we briefly introduce rectified power unit (RePU) activation function, which is the power of the ReLU activation functions. The RePU function is defined as follows:
	\begin{equation*}
		\sigma_{s}(x)= \begin{cases}x^{s}, & x \geq 0, \\ 0, & x<0,\end{cases}
	\end{equation*}
	where $s$ is a non-negative integer. The RePU function with $s=1$ is the commonly used ReLU function $\sigma_{1}$. We denote $\sigma_{2}, \sigma_{3}$ ReQU and ReCU for $s=2,3$. 
	
	Deep neural networks with the ReLU activation function are becoming increasingly popular because of their high efficiency and versatility. In this study, we performed numerous matrix multiplications based on scalar multiplication and addition. The work \cite{yarotsky2017error} has revealed that $x^{2}, xy$ can be approximated using a ReLU neural network with depth, the number of weights, and computation units of order $\mathcal{O}(\log \frac{1}{\epsilon})$. However, studies have indicated that deep neural networks using RePUs ($s \ge 2$ ) as the activation functions exhibit superior approximation property for smooth functions than those using ReLUs \cite{li2019better, opschoor2021exponential}. By replacing the ReLU with the RePU($s \ge 2$), the bivariate functions $x^2$ and $xy$ can be represented with no approximation error using networks with a few nodes and nonzero weights. To be more precise, we introduce the following lemma (see \cite{li2019better})
	
	\begin{lem}\label{x}
		For any $x, y \in \mathbb{R}$, the following identities hold
		
		\begin{equation*}
			x^{2} =\mathbf{\beta}_{2}^{T} \sigma_{2}\left(\mathbf{\omega}_{2} x\right),
		\end{equation*}
		\begin{equation*}
			x =\mathbf{\beta}_{1}^{T} \sigma_{2}\left(\mathbf{\omega}_{1} x+\mathbf{\gamma}_{1}\right), 
		\end{equation*}
		\begin{equation*}
			x y =\mathbf{\beta}_{1}^{T} \sigma_{2}\left(\mathbf{\omega}_{1} x+\mathbf{\gamma}_{1} y\right),
		\end{equation*}
		where
		$$
		\mathbf{\beta}_{2}=[1,1]^{T}, \mathbf{\omega}_{2}=[1,-1]^{T}, \mathbf{\beta}_{1}=\frac{1}{4}[1,1,-1,-1]^{T}, \mathbf{\omega}_{1}=[1,-1,1,-1]^{T}, \mathbf{\gamma}_{1}=[1,-1,-1,1]^{T}.
		$$
		and $\sigma_{2}$ acts componentwise.
	\end{lem}

	\section{ReQU-Based Neural Network Calculus}
	
	\subsection{Basic Definitions and Operations}
	For convenience, we follow the notations of neural networks presented in \cite{Elbrachter2018DNNER, kutyniok2021theoretical}. Next, we introduce concatenation and parallelization operations for neural networks, which may be used to construct complex neural networks from simple networks. Furthermore, we derive the complexity estimates of operations of ReQU neural networks.
	\begin{defn}
		Let $n, L\in \mathbb{N}$. A neural network $\Phi$ with input dimension $\mathrm{dim}_{\mathrm{in}}(\Phi) = n$ and number of layers $L(\Phi)=L$ is a matrix-vector sequence
		$$
		\Phi=\left(\left(\textbf{A}_{1}, \textbf{b}_{1}\right), \cdots,\left(\textbf{A}_{L}, \textbf{b}_{L}\right)\right),
		$$
		where $\textbf{A}_{k}$ are $N_{k} \times N_{k-1}$ matrices, and $\textbf{b}_{k} \in \mathbb{R}^{N_{k}}, N_{0}=n, N_{1}, \cdots, N_{L} \in \mathbb{N}$. Let $m\in \mathbb{N}$ and $\rho: \mathbb{R} \rightarrow \mathbb{R}$ be an arbitrary activation function. We define
		\begin{equation*}
			\mathrm{R}_{\rho}(\Phi): \mathbb{R}^{n} \rightarrow \mathbb{R}^{N_{L}}, \quad R_{\rho}(\Phi)(\textbf{x})=\textbf{x}_{L},
		\end{equation*}
		where $\textbf{x}_{L}$ is expressed as follows:
		$$
		\left\{\begin{array}{l}
			\textbf{x}_{0}:=\textbf{x}, \\
			\textbf{x}_{k}:=\rho\left(\textbf{A}_{k} \textbf{x}_{k-1}+\textbf{b}_{k}\right), \quad k=1,2, \ldots, L-1, \\
			\textbf{x}_{L}:=\textbf{A}_{L} \textbf{x}_{L-1}+\textbf{b}_{L},
		\end{array}\right.
		$$
		and
		$$
		\rho(\boldsymbol{y}):=\left(\rho\left(y^{1}\right), \cdots, \rho\left(y^{m}\right)\right)^{T}, \quad \forall \boldsymbol{y}=\left(y^{1}, \cdots, y^{m}\right)^{T} \in \mathbb{R}^{m}.
		$$
	\end{defn}
	We denote $N(\Phi):=n+\sum_{k=1}^{L} N_{j}$ number of nodes, $M_{k}(\Phi):=\left\|\mathbf{A}_k\right\|_{0}+\left\|\mathbf{b}_k\right\|_{0}$ the number of nonzero weights in $k$-th layer for $k \leq L$ and $M(\Phi):=\sum_{k=1}^{L} M_{k}(\Phi)$ the total number of nonzero weights of $\Phi$. Here, $\left\|\textbf{A}\right\|_{0}$ denotes the number of nonzero entries of matrix A.
	We use the number of layers, number of nodes, and number of nonzero weights to measure the complexity of the neural networks.
	
	Next, we define the concatenation of two neural networks as follows:
	\begin{defn}\label{neural networkdefinition}
		Let $L_{1}, L_{2} \in \mathbb{N}$ and let 
		$\Phi^{1}=\left(\left(\textbf{A}_{1}^{1}, \textbf{b}_{1}^{1}\right), \ldots,\left(\textbf{A}_{L_{1}}^{1}, \textbf{b}_{L_{1}}^{1}\right)\right)$, $\Phi^{2}=\left(\left(\textbf{A}_{1}^{2}, \textbf{b}_{1}^{2}\right), \ldots\right.$, $\left.\left(\textbf{A}_{L_{2}}^{2}, \textbf{b}_{L_{2}}^{2}\right)\right)$ be two neural networks with activation function $\rho$. The input dimension of $\Phi^{1}$ is the same as the output dimension of $\Phi^{2}$. Next, we denote $\Phi^{1} \bullet \Phi^{2}$ as the concatenation of $\Phi_1$, $\Phi_2$ as follows:
		$$
		\Phi^{1} \bullet \Phi^{2}:=\left(\left(\textbf{A}_{1}^{2}, \textbf{b}_{1}^{2}\right), \ldots,\left(\textbf{A}_{L_{2}-1}^{2}, \textbf{b}_{L_{2}-1}^{2}\right),\left(\textbf{A}_{1}^{1} \textbf{A}_{L_{2}}^{2}, \textbf{A}_{1}^{1} \textbf{b}_{L_{2}}^{2}+\textbf{b}_{1}^{1}\right),\left(\textbf{A}_{2}^{1}, \textbf{b}_{2}^{1}\right), \ldots,\left(\textbf{A}_{L_{1}}^{1}, \textbf{b}_{L_{1}}^{1}\right)\right),
		$$
		where $L(\Phi^{1} \bullet \Phi^{2}) = L_1 + L_2 -1$.
	\end{defn}
	
	Next, a lemma, which shows that $M\left(\Phi^{1} \bullet \Phi^{2}\right)$ can be estimated by $\max \left\{M\left(\Phi^{1}\right), M\left(\Phi^{2}\right)\right\}$ in some special cases, is presented (see also \cite{kutyniok2021theoretical}).
	\begin{lem}\label{1}
		Let $\Phi$ be a neural network with $m$ dimensional output and $d$ dimensional input. If $\mathbf{a} \in \mathbb{R}^{1 \times m}$, then, for all $k=1, \ldots, L(\Phi)$,
		$$
		M_{k}(((\mathbf{a}, 0)) \bullet \Phi) \leq M_{k}(\Phi).
		$$
		In particular, $M((\mathbf{a}, 0) \bullet \Phi) \leq M(\Phi) .$ Moreover, if $\mathbf{D} \in \mathbb{R}^{d \times n}$ such that, for every $k \leq d$ there is at most one $l_{k} \leq n$, such that $\mathbf{D}_{k, l_{k}} \neq 0$, then, for all $k=1, \ldots, L(\Phi)$,
		$$
		M_{k}\left(\Phi \bullet\left(\left(\mathbf{D}, \mathbf{0}_{\mathbb{R}^{d}}\right)\right)\right) \leq M_{k}(\Phi).
		$$
		In particular, it holds that $M\left(\Phi \bullet\left(\left(\mathbf{D}, \mathbf{0}_{\mathbb{R}^{d}}\right)\right)\right) \leq M(\Phi)$.
	\end{lem}
	In general, there is no bound on $M\left(\Phi^{1} \bullet \Phi^{2}\right)$ that is linear in $M\left(\Phi^{1}\right)$ and $M\left(\Phi^{2}\right)$. To overcome this problem, we introduce an alternative concatenation to control the number of nonzero weights. Before proceeding, we detail the following lemma, constructing the ReQU neural network of the identity function based on Definition \ref{neural networkdefinition} and Lemma \ref{x}. Hereafter, we take $\sigma_{2}$ as the activation function.
	\begin{lem}\label{ID}
		For any $n, L \in \mathbb{N}$ and $L\geq 2, $ there exists a ReQU neural network $ \Phi_{n, L}^{\mathbf{Id}}$ with n dimensional input and output, satisfying
		$$
		\mathrm{R}_{\sigma_{2}}\left(\Phi_{n, L}^{\mathbf{Id}}\right)= \mathbf{Id}_{\mathbb{R}^{n}}
		$$
		and
		$$M(\Phi_{n, L}^{\mathbf{Id}})= 20nL - 28n,$$
		where $\mathbf{Id}_{\mathbb{R}^{n}}$ is identity matrix.
		Moreover, let $\mathbf{\omega}_1 =(1,-1,1,-1)^T$, $\mathbf{\gamma}_{1} =(1,-1,1,1)^T$, $\mathbf{\beta}_1 =\frac{1}{4}(1,1,-1,-1)^T,$ and denote $\mathbf{W}\in \mathbb{R}^{4n \times n}, \mathbf{\Gamma}\in\mathbb{R}^{4n \times 1}, \mathbf{B}\in \mathbb{R}^{n \times 4n}$ as follows :
		\begin{equation*}
			\mathbf{W}:=\begin{pmatrix}
				\mathbf{\omega}_1  &0 &\cdots  &0\\ 
				0& \mathbf{\omega}_1 &\cdots &0 \\ 
				\vdots& \vdots & \ddots&\vdots\\
				0&0&\cdots&\mathbf{\omega}_1
			\end{pmatrix}, \  
			\mathbf{\Gamma}:=\begin{pmatrix}
				\mathbf{\gamma}_{1}\\ 
				\mathbf{\gamma}_{1}\\ 
				\vdots\\ 
				\mathbf{\gamma}_{1}
			\end{pmatrix}, \
			\mathbf{B}:=\begin{pmatrix}
				\mathbf{\beta}_1^T  &0 &\cdots  &0\\ 
				0& \mathbf{\beta}_1^T &\cdots &0 \\ 
				\vdots& \vdots & \ddots&\vdots\\
				0&0&\cdots&\mathbf{\beta}_1^T
			\end{pmatrix},
		\end{equation*}
		Then $\Phi_{n, L}^{\mathbf{Id}} $ can be expressed by the following expression:
		\begin{equation}\label{IDL}
			\Phi_{n, L}^{\mathbf{Id}}=(\left(\mathbf{W}, \Gamma\right), \underbrace{\left(\mathbf{WB}, \Gamma\right), \ldots,\left(\mathbf{WB}, \Gamma\right)}_{L-2 \text { times }}, \left(\mathbf{B},0 \right)).
		\end{equation}
		For $L = 1$, we have the following expression:
		\begin{equation}\label{IDD}
			\Phi_{n, 1}^{\mathbf{Id}}=\left(\left(\mathbf{Id}_{\mathbb{R}^{n}}, 0\right)\right)
		\end{equation}
		and $M(\Phi_{n, 1}^{\mathbf{Id}}) = n$.
	\end{lem}
	\begin{proof}
		When L = 1, the proof is true. When $L \in \mathbb{N} \cap[2, \infty)$, for every $n \in \mathbb{N}$, based on \eqref{IDL}, we have the following expression:
		\begin{equation*}
			\begin{aligned}
				M(\Phi_{n, L}^{\mathbf{Id}}) &= \|\mathbf{W}\|_0 + \|\mathbf{\Gamma}\|_0 + \|\mathbf{B}\|_0 + (L-2)(\|\mathbf{WB}\|_0+\|\mathbf{\Gamma}\|_0)\\
				&= 12n + (L-2)(16n + 4n)\\
				& = 20nL - 28n.
			\end{aligned} 
		\end{equation*}The proof is completed.
	\end{proof}

	The sparse concatenation of two neural networks is introduced below. 
	
	\begin{defn}
		Let $\Phi^{1} $ and $\Phi^{2}$ be two ReQU neural networks, such that the output dimension of $\Phi^{2}$ and input dimension of $\Phi^{1}$ equal $n$. We define the sparse concatenation of $\Phi^{1}$ and $\Phi^{2}$ as
		$$
		\Phi^{1} \odot \Phi^{2}:=\Phi^{1} \bullet \Phi_{n, 2}^{\mathbf{Id}} \bullet \Phi^{2}.
		$$
	\end{defn}
	
	We now introduce another operation, which we call parallelization.
	
	\begin{defn}
		Let $\Phi^{1}, \ldots, \Phi^{k}$ be neural networks that have equal input dimension, such that $\Phi^{i}=\left(\left(\mathbf{A}_{1}^{i}, \mathbf{b}_{1}^{i}\right), \ldots,\left(\mathbf{A}_{L}^{i}, \mathbf{b}_{L}^{i}\right)\right)$ holds for some $L \in \mathbb{N} .$ Next, we define the parallelization of $\Phi^{1}, \ldots, \Phi^{k}$ with
		\begin{equation}\label{PK}
			\mathrm{P}\left(\Phi^{1}, \ldots, \Phi^{k}\right):=\left(\left(\left(\begin{array}{cccc}
				\mathbf{A}_{1}^{1} &0 &\cdots  &0 \\
				0& \mathbf{A}_{1}^{2} &\cdots &0 \\
				\vdots& \vdots & \ddots&\vdots\\
				0&0&\cdots& \mathbf{A}_{1}^{k}
			\end{array}\right),\left(\begin{array}{c}
				\mathbf{b}_{1}^{1} \\
				\mathbf{b}_{1}^{2} \\
				\vdots \\
				\mathbf{b}_{1}^{k}
			\end{array}\right)\right), \ldots,\left(\left(\begin{array}{cccc}
				\mathbf{A}_{L}^{1} &0 &\cdots  &0  \\
				0& \mathbf{A}_{L}^{2} &\cdots &0  \\
				\vdots& \vdots & \ddots&\vdots\\
				0&0&\cdots& \mathbf{A}_{L}^{k}
			\end{array}\right),\left(\begin{array}{c}
				\mathbf{b}_{L}^{1} \\
				\mathbf{b}_{L}^{2} \\
				\vdots \\
				\mathbf{b}_{L}^{k}
			\end{array}\right)\right)\right).
		\end{equation}
		Now, let $\Phi$ be a neural network and $L \in \mathbb{N}$ such that $L(\Phi) \leq L .$ Next, define the neural network
		\begin{equation}\label{EL}
			E_{L}(\Phi):=\left\{\begin{array}{ll}
				\Phi, & \text { if } L(\Phi)=L, \\
				\Phi_{ \operatorname{dim}_{\text {out }}(\Phi), L-L(\Phi) }^{\mathbf{Id}}\odot \Phi, & \text { if } L(\Phi)<L.
			\end{array}\right.
		\end{equation}
		Finally, let $\tilde{\Phi}^{1}, \ldots, \tilde{\Phi}^{k}$ be neural networks that have the same input dimension, and let
		\begin{equation*} 
			\tilde{L}:=\max \left\{L\left(\tilde{\Phi}^{1}\right), \ldots, L\left(\tilde{\Phi}^{k}\right)\right\}.
		\end{equation*}
		Next, we define
		\begin{equation}\label{P}
			\mathrm{P}\left(\tilde{\Phi}^{1}, \ldots, \tilde{\Phi}^{k}\right):=\mathrm{P}\left(E_{\tilde{L}}\left(\tilde{\Phi}^{1}\right), \ldots, E_{\tilde{L}}\left(\tilde{\Phi}^{k}\right)\right).
		\end{equation}
		We call $\mathrm{P}\left(\tilde{\Phi}^{1}, \ldots, \tilde{\Phi}^{k}\right)$ the parallelization of $\tilde{\Phi}^{1}, \ldots, \tilde{\Phi}^{k}$.
	\end{defn}
	
	The following two lemmas provide the properties of the sparse concatenation and the parallelization of neural networks. 
	
	\begin{lem}\label{11}
		Let $\Phi^{1}, \Phi^{2}$ be  ReQU neural networks. The input dimension of $\Phi^{1}$ equals the output dimension of $\Phi^{2}$, then, for the neural network $\Phi^{1} \odot \Phi^{2}$ it holds
		\begin{enumerate}[(i)]
			\item $[\mathrm{R}_{\sigma_{2}}\left(\Phi^{1}\right) \circ \mathrm{R}_{\sigma_{2}}\left(\Phi^{2}\right)] (\textbf{x})= [\mathrm{R}_{\sigma_{2}}\left(\Phi^{1} \odot \Phi^{2}\right)] (\mathbf{x}),$ for every $\textbf{x} \in \mathbb{R}^{\mathrm{dim}_{\mathrm{in }}\left(\Phi^{2}\right)}$,
			\item $L\left(\Phi^{1} \odot \Phi^{2}\right) = L\left(\Phi^{1}\right)+L\left(\Phi^{2}\right)$,
			\item $M\left(\Phi^{1} \odot \Phi^{2}\right) \leq M\left(\Phi^{1}\right)+M\left(\Phi^{2}\right)+4M_{1}\left(\Phi^{1}\right)+4M_{L\left(\Phi^{2}\right)}\left(\Phi^{2}\right) +4 	\mathrm{dim}_{\mathrm{out }}(\Phi^{2})\leq 5 M\left(\Phi^{1}\right)+5 M\left(\Phi^{2}\right) + 4 	\mathrm{dim}_{\mathrm{out }}(\Phi^{2}) ,$
			\item $M_{1}\left(\Phi^{1} \odot \Phi^{2}\right)=M_{1}\left(\Phi^{2}\right)$, if $L(\Phi_{2})\geq2$,
			\item $M_{L\left(\Phi^{1} \odot \Phi^{2}\right)}\left(\Phi^{1} \odot \Phi^{2}\right)=M_{L\left(\Phi^{1}\right)}\left(\Phi^{1}\right)$, if $L(\Phi_{2})\geq2$. 
		\end{enumerate}
	\end{lem}
	\begin{proof}
		For every $i \in\{1,2\}$, let $L_{i} , N_{1}^{i}, N_{2}^{i}, \ldots, N_{L_{i}}^{i}\in \mathbb{N},  \left(\mathbf{A}_{k}^{i}, \mathbf{b}_{k}^{i}\right) \in \mathbb{R}^{N_{k}^{i} \times N_{k-1}^{i}} \times \mathbb{R}^{N_{k}^{i}}, k\in$ $\left\{1,2, \ldots, L_{i}\right\}$, such that 
		$$\Phi^{i}=\left(\left(\mathbf{A}_{1}^{i}, \mathbf{b}_{1}^{i}\right), \ldots,\left(\mathbf{A}_{L_{i}}^{i}, \mathbf{b}_{L_{i}}^{i}\right)\right).$$
		For $k \in \left\{1,2, \ldots, L_{1}+L_{2}\right\}$, let $\left(\mathbf{A}_{k}, \mathbf{b}_{k}\right) \in \mathbb{R}^{N_{k} \times N_{k-1}} \times \mathbb{R}^{N_{k}}$ be the matrix-vector tuples that satisfy $$\Phi^{1} \odot \Phi^{2}=\left(\left(\mathbf{A}_{1}, \mathbf{b}_{1}\right), \ldots,\left(\mathbf{A}_{L_{1}+L_{2}}, \mathbf{b}_{L_{1}+L_{2}}\right)\right).$$ 
		Furthermore, for every $\mathbf{x} \in \mathbb{R}^{N_{0}}$, let $\mathbf{x}_k$ be expressed as follows:
		\begin{equation}\label{3.8}
			\mathbf{x}_{k}(\mathbf{x})= \begin{cases}\sigma_2 \left(\mathbf{A}_{1} \mathbf{x}+\mathbf{b}_{1}\right), & \text { if } k=1,\\ 
				\sigma_2 \left(\mathbf{A}_{k} \mathbf{x}_{k-1}(\mathbf{x})+\mathbf{b}_{k}\right), & \text { if }1<k<L_{1}+L_{2}, \\ 
				\mathbf{A}_{L_1 + L_2} \mathbf{x}_{L_1 + L_2-1}(\mathbf{x})+\mathbf{b}_{L_1 + L_2}, & \text { if } k=L_{1}+L_{2}.\end{cases}
		\end{equation}
		By definition, we have the following expression:
		\begin{equation}\label{3.9}
			\Phi^{1} \odot    \Phi^{2}:=\left(\left(\mathbf{A}_{1}^{2}, \mathbf{b}_{1}^{2}\right), \ldots,\left(\mathbf{A}_{L_{2}-1}^{2}, \mathbf{b}_{L_{2}-1}^{2}\right),\left(\mathbf{W} \mathbf{A}_{L_{2}}^{2}, \mathbf{W} \mathbf{b}_{L_{2}}^{2}+\mathbf{\Gamma}\right),\left(\mathbf{A}_{1}^{1}\mathbf{B}, \mathbf{b}_{1}^{1}\right), \left(\mathbf{A}_{2}^{1}, \mathbf{b}_{2}^{1}\right), \ldots,\left(\mathbf{A}_{L_{1}}^{1}, \mathbf{b}_{L_{1}}^{1}\right)\right).
		\end{equation}
		For every $k \in\left\{1,2, \ldots, L_{2}-1\right\}$, $\left(A_{k}, b_{k}\right)=\left(A_{k}^{2}, b_{k}^{2}\right) .$ This implies that 
		\begin{equation}\label{3.10}
			\mathbf{A}_{L_{2}}^{2} \mathbf{x}_{L_{2}-1}(\mathbf{x})+\mathbf{b}_{L_{2}}^{2}=\left[R_{\sigma_{2}}\left(\Phi^{2}\right)\right](\mathbf{x}).
		\end{equation}
		Thus, we have
		\begin{equation}\label{L2}
			\mathbf{x}_{L_{2}}(\mathbf{x}) = \sigma_{2}\left(\textbf{A}_{L_{2}} \mathbf{x}_{L_{2}-1}(\mathbf{x})+\textbf{b}_{L_{2}}\right)=\sigma_2 \left(\mathbf{W} \mathbf{A}_{L_{2}}^{2} \mathbf{x}_{L_{2}-1}(\mathbf{x})+(\mathbf{W} \mathbf{b}_{L_{2}}^{2}+\mathbf{\Gamma})\right).
		\end{equation}
		Next \eqref{3.9},\eqref{3.10},  \eqref{L2}, and Lemma \ref{ID} ensure that, for every $\mathbf{x} \in \mathbb{R}^{N_{0}}$
		\begin{equation*}
			\begin{aligned}
				\mathbf{x}_{L_{2}+1}(\mathbf{x}) &=\sigma_{2}(\textbf{A}_{L_{2}+1}\left(\sigma_{2}\left(\textbf{A}_{L_{2}} \mathbf{x}_{L_{2}-1}(\mathbf{x})+\textbf{b}_{L_{2}}\right)\right)+\textbf{b}_{L_2 + 1})\\
				&=\sigma_2(\mathbf{A}_{1}^{1}\mathbf{B}\sigma_2 \left(\mathbf{W} \mathbf{A}_{L_{2}}^{2} \mathbf{x}_{L_{2}-1}(\mathbf{x})+(\mathbf{W} \mathbf{b}_{L_{2}}^{2}+\mathbf{\Gamma})\right)
				+\mathbf{b}_{1}^{1}) \\
				&=\sigma_2(\mathbf{A}_{1}^{1}\mathbf{B}\sigma_{2}(\mathbf{W}\left[R_{\sigma_{2}}\left(\Phi^{2}\right)\right](\mathbf{x})+\mathbf{\Gamma})+\mathbf{b}_{1}^{1})\\ 
				&=\sigma_2(\mathbf{A}_{1}^{1}\left[R_{\sigma_{2}}\left(\Phi^{2}\right)\right](\mathbf{x})+\mathbf{b}_{1}^{1}) \\
			\end{aligned}
		\end{equation*} 
		holds. Combining this with \eqref{3.8}, we obtain (i). Moreover, (ii), (iv), and (v) follow directly from \eqref{3.9}.  
		
		To prove (iii), note that
		$$M_{L_2}(\Phi^{1} \odot    \Phi^{2})\leq 4 M_{L_2}(\Phi^2) + 4 \mathrm{dim}_{\mathrm{out}}(\Phi^2)$$
		$$M_{L_2 + 1}(\Phi^1 \odot \Phi^2) \leq 4M_1(\Phi^1),$$
		then, by \eqref{3.9} we have
		\begin{align*}
			M(\Phi^1 \odot \Phi^2)&\leq M(\Phi^1) + M(\Phi^2) + M_{L_2}(\Phi^{1} \odot    \Phi^{2}) + M_{L_2 + 1}(\Phi^1 \odot \Phi^2)\\
			&\leq M(\Phi^1) + M(\Phi^2) + 4M_1(\Phi^1) +4 M_{L_2}(\Phi^2) + 4 \mathrm{dim}_{\mathrm{out}}(\Phi^2) .
		\end{align*}
		This completes the proof of Lemma \ref{11}.
	\end{proof}
	\begin{lem}\label{11b} Let $\Phi^{1}, \ldots, \Phi^{k}$ be ReQU neural networks. If the input dimension of $\Phi^{i}$ equals the input dimension of $\Phi^{j}$ for all $i, j$, then, for the neural network $ \mathrm{P}\left(\Phi^{1}, \Phi^{2}, \ldots, \Phi^{k}\right)$, we have the following expression:
		\begin{enumerate}[(i)]
			\item$\mathrm{R}_{\sigma_{2}}\left(\mathrm{P}\left(\Phi^{1}, \Phi^{2}, \ldots, \Phi^{k}\right)\right)\left(\mathbf{x}_{1}, \ldots, \mathbf{x}_{k}\right) = \left(\mathrm{R}_{\sigma_{2}}\left(\Phi^{1}\right)\left(\mathbf{x}_{1}\right), \mathrm{R}_{\sigma_{2}}\left(\Phi^{2}\right)\left(\mathbf{x}_{2}\right), \ldots, \mathrm{R}_{\sigma_{2}}\left(\Phi^{k}\right)\left(\mathbf{x}_{k}\right)\right),$ for all $\mathbf{x}_{1}, \ldots \mathbf{x}_{k} \in R^{n},$ 
			\item$L\left(\mathrm{P}\left(\Phi^{1}, \Phi^{2}, \ldots, \Phi^{k}\right)\right) = \max _{i=1, \ldots, k} L\left(\Phi^{i}\right)$,
			\item $M\left(\mathrm{P}\left(\Phi^{1}, \Phi^{2}, \ldots, \Phi^{k}\right)\right) \leq \sum_{i=1}^{k}\left( M\left(\Phi^{i}\right) + 4 M_{L\left(\Phi^{i}\right)}\left(\Phi^{i}\right)
			+\mathrm{dim}_{\mathrm{out}}(\Phi^{i})\big(20L(\mathrm{P}\left(\Phi^{1}, \Phi^{2}, \ldots, \Phi^{k}\right))+8 \big) \right),$
			\item $M\left(\mathrm{P}\left(\Phi^{1}, \Phi^{2}, \ldots, \Phi^{k}\right)\right)=\sum_{i=1}^{k} M\left(\Phi^{i}\right)$, if $L\left(\Phi^{1}\right) = L\left(\Phi^{2}\right) = \ldots = L\left(\Phi^{k}\right)$,
			\item $M_{1}\left(\mathrm{P}\left(\Phi^{1}, \Phi^{2}, \ldots, \Phi^{k}\right)\right)=\sum_{i=1}^{k} M_{1}\left(\Phi^{i}\right),$
			\item $M_{L\left(\mathrm{P}\left(\Phi^{1}, \Phi^{2}, \ldots, \Phi^{k}\right)\right)}\left(\mathrm{P}\left(\Phi^{1}, \Phi^{2}, \ldots, \Phi^{k}\right)\right) \leq \sum_{i=1}^{k} \max \left\{4 \mathrm{dim}_{\mathrm{out }}\left(\Phi^{i}\right), M_{L\left(\Phi^{i}\right)}\left(\Phi^{i}\right)\right\},$
			\item $M_{L\left(\mathrm{P}\left(\Phi^{1}, \Phi^{2}, \ldots, \Phi^{k}\right)\right)}\left(\mathrm{P}\left(\Phi^{1}, \Phi^{2}, \ldots, \Phi^{k}\right)\right)=\sum_{i=1}^{k} M_{L\left(\Phi^{i}\right)}\left(\Phi^{i}\right)$, if $L\left(\Phi^{1}\right)=L\left(\Phi^{2}\right)=\ldots=L\left(\Phi^{k}\right).$
		\end{enumerate}
		
	\end{lem}
	
	\begin{proof}
		By Lemma \ref{11}, for every $i \in\{1,2, \ldots, n\}$ we have the following expression:
		$$
		R_{\sigma_{2}}\left(E_{\tilde{L}}\left(\Phi^{i}\right)\right)=R_{\sigma_{2}}\left(\Phi^{i}\right).
		$$
		Combining this with \eqref{EL} and \eqref{P} establishes (i). Furthermore, (ii), (iv), (v), and (vii) follow directly from \eqref{PK}, \eqref{EL}, and \eqref{P}. In addition, note that \eqref{IDL}, \eqref{ID}, and \eqref{EL} ensure for every $i \in\{1,2, \ldots, k\}$ that
		$$
		M_{L(\Phi)}\left(E_{L(\Phi)}\left(\Phi^{i}\right)\right) \leq \max \left\{4 \operatorname{dim}_{\text {out }}\left(\Phi^{i}\right), M_{L\left(\Phi^{i}\right)}\left(\Phi^{i}\right)\right\},
		$$
		which implies $(vi).$ 
		
		Now, we need to show $(iii).$ Let $\Phi = \mathrm{P}\left(\Phi^{1}, \Phi^{2}, \ldots, \Phi^{k}\right)$. By Lemma \ref{ID}, Lemma \ref{11}, and \eqref{EL}, for every $i \in\{1,2, \ldots, k\}$, if $L(\Phi)-L(\Phi^i) \ge 2$, we have the following expression:
		\begin{align}\label{weight-Phi-j}
			M\left(E_{L(\Phi)}\left(\Phi^{i}\right)\right) 
			& \leq  M\left(\Phi_{\mathrm{dim}_{\text {out }}(\Phi^i), L(\Phi)-L(\Phi^i)}^{\mathrm{Id}}\right)+ M\left(\Phi^{i}\right)+4M_1 \left(\Phi_{\mathrm{dim}_{\text {out }}(\Phi^i), L(\Phi)-L(\Phi^i)}^{\mathrm{Id}}\right) \nonumber\\
			&\quad +4M_{L\left(\Phi^{i}\right)}\left(\Phi^{i}\right) + 4\mathrm{dim}_{\text {out }}(\Phi^i) \nonumber\\
			&\leq 20 \mathrm{dim}_{\text {out }}(\Phi^i) L(\Phi) - 28\mathrm{dim}_{\text {out }}(\Phi^i) + M(\Phi^i)+32 \mathrm{dim}_{\text {out }}(\Phi^i) \nonumber\\
			&\quad +4M_{L\left(\Phi^{i}\right)}\left(\Phi^{i}\right) + 4\mathrm{dim}_{\text {out }}(\Phi^i) \nonumber\\
			&= M(\Phi^i) + 4M_{L\left(\Phi^{i}\right)}\left(\Phi^{i}\right) + (20 L(\Phi)+ 8)\mathrm{dim}_{\text {out }}(\Phi^i).
		\end{align}
		For  $L(\Phi)-L(\Phi^i) = 1$, by Lemma \ref{ID} and a simple calculation, \eqref{weight-Phi-j} still holds. Combining this with (iv) implies (iii). 
		The proof of Lemma \ref{11b} is thus completed.
	\end{proof}

	\section{Complexity Bound of ReQU Neural Network to Approximate Matrix Inversion}
	To proceed, we first introduce vectorized matrix hereinafter to stay in the classical neural network setup.
	
	\begin{defn}
		Let $\mathbf{A} \in \mathbb{R}^{d \times l}$. We denote
		$$
		\operatorname{vec}(\mathbf{A}):=\left(\mathbf{A}_{1,1}, \ldots, \mathbf{A}_{d, 1}, \ldots, \mathbf{A}_{1, l}, \ldots, \mathbf{A}_{d, l}\right)^{T} \in \mathbb{R}^{d l}
		$$
		Moreover, for a vector $\mathbf{v}=\left(\mathbf{v}_{1,1}, \ldots, \mathbf{v}_{d, 1}, \ldots, \mathbf{v}_{1, d}, \ldots, \mathbf{v}_{d, l}\right)^{T} \in \mathbb{R}^{d l}$, we set
		$$
		\operatorname{matr}(\mathbf{v}):=\left(\mathbf{v}_{i, j}\right)_{i=1, \ldots, d, j=1, \ldots, l} \in \mathbb{R}^{d \times l}
		$$
	\end{defn}
	In this section, we prove the following theorem.
	\begin{thm}\label{3.4}
		For $\epsilon, \delta \in(0,1)$, we define
		$$
		l = l(\epsilon, \delta):=\left\lceil\log_{2}(\log_{1-\delta}(\delta\epsilon)+1)\right\rceil,
		$$
		where $\lceil a\rceil:=\min \{b \in \mathbb{Z}: b \geq a\} $ for $a \in \mathbb{R}$. Let $d\in \mathbb{N}$,  there exists a ReQU neural network $\Phi_{\mathrm{inv} ; \epsilon}^{d}$ with $d^{2}$ dimensional input and $d^{2}$ dimensional output satisfying the following properties:
		\begin{enumerate}[(i)]
			\item$\sup_{\mathbf{A}\in\mathbb{R}^{d\times d},\|\mathbf{A}\|_2\leq 1-\delta}\left\|\left(\mathbf{I} \mathbf{d}_{\mathbb{R}^{d}}-\mathbf{A}\right)^{-1}-\operatorname{matr}\left(\mathrm{R}_{\sigma_{2}}\left(\Phi_{\mathrm{inv}; \epsilon}^{d}\right)(\operatorname{vec}(\mathbf{A}))\right)\right\|_{2} \leq \epsilon,$
			\item $L\left(\Phi_{\mathrm {inv};\epsilon}^{d}\right) = 2l+1, $
			\item there exists a universal constant $C_{\mathrm{inv}}>0$ such that
			$$M\left(\Phi_{\mathrm{inv} ; \epsilon}^{d}\right) \leq C_{\mathrm{inv}}d^3l^2.$$ 
		\end{enumerate}
		
	\end{thm}
	Here, $\|\cdot\|_{2}$ denotes the spectral norm of a matrix. The main purpose of our proof is based on the fact that \textit{Neumann series}  $\sum_{k=0}^{m} \mathbf{A}^{k}$ converges exponentially fast to $\left(\mathbf{I d}_{\mathbb{R}^{d}}-\mathbf{A}\right)^{-1}$ as $m \rightarrow \infty$, for $\mathbf{A} \in \mathbb{R}^{d \times d}$, satisfying $\|\mathbf{A}\|_{2} \leq 1-\delta$ for some $\delta \in(0,1)$. To simplify the structure of the ReQU neural network, we consider $m=2^{l-1}$, and then, factorize polynomial $\sum_{k=0}^{2^l-1} \mathbf{A}^{k} = \prod_{k=0}^{l-1}(\textbf{A}^{2^k}+\mathbf{Id}_{\mathbb{R}^{d}})$. This reduces the problem to use a neural network to approximate the matrix multiplication.
	
	\subsection{ReQU Neural Network to Represent Matrix Multiplication}
	
	\begin{prop}\label{p}
		Let $d, n, l \in \mathbb{N}.$  There exists a ReQU neural network $\Phi_{\mathrm{mult} }^{d, n, l}$ with $n\cdot(d+l)$ dimensional input and $dl$ dimensional output, such that it satisfies the following properties:
		\begin{enumerate}[(i)]
			\item $\operatorname{matr}\left(\mathrm{R}_{\sigma_{2}}\left(\Phi_{\mathrm{mult}}^{d, n, l}\right)(\operatorname{vec}(\mathbf{A}), \operatorname{vec}(\mathbf{B}))\right)=\mathbf{A B}, $ for any $\mathbf{A}\in\mathbb{R}^{d\times n }, \mathbf{B}\in \mathbb{R}^{n\times l}$,
			\item $L\left(\Phi_{\mathrm{mult}}^{d, n, l}\right) =2 $,
			\item $M\left(\Phi_{\mathrm{mult}}^{d, n, l}\right) \leq 12 d n l$,
			\item $M_{1}\left(\Phi_{\mathrm{mult}}^{d, n, l}\right) \leq 8 d n l, \quad$ as well as $\quad M_{L\left(\Phi_{\mathrm{mult}}^{d, n, l}\right)}\left(\Phi_{\text {mult }}^{d, n, l}\right) \leq 4 d n l$.
		\end{enumerate}
	\end{prop}
	\begin{proof}
		By Lemma \ref{x}, we can realize scalar multiplication with a ReQU neural network $\Phi$, satisfying the following equation:
		\begin{equation}\label{2}
			x y=R_{\sigma_2}(\Phi)(x,y)=\mathbf{\beta}_{1}^{T} \sigma_{2}\left(\mathbf{\omega}_{1} x+\mathbf{\gamma}_{1} y\right)
		\end{equation}
		and
		\begin{align}
			L\left(\Phi\right) &= 2,\label{3}\\
			M\left(\Phi\right) &\leq 12,\label{4}\\
			M_1\left(\Phi\right) &\leq 8,\label{5}\\
			M_{L\left(\Phi\right)}\left(\Phi\right) &\leq 4.\label{6}
		\end{align}
		
		First, for $i \in\{1, \ldots, d\}, k \in\{1, \ldots, n\}, j \in\{1, \ldots, l\}$, we define matrix $\mathbf{D}_{i, k, j}$ such that
		$$
		\mathbf{D}_{i, k, j}(\operatorname{vec}(\mathbf{A}), \operatorname{vec}(\mathbf{B}))=\left(\mathbf{A}_{i, k}, \mathbf{B}_{k, j}\right), \quad \forall  \mathbf{A} \in \mathbb{R}^{d \times n}, \mathbf{B} \in \mathbb{R}^{n \times l},
		$$
		and denote
		$$
		\Phi_{i, k, j }:=\Phi \bullet\left(\left(\mathbf{D}_{i, k, j}, \mathbf{0}_{\mathbb{R}^{2}}\right)\right).
		$$
		Then, by \eqref{2}, we have the following expression:
		\begin{equation}\label{66}
			\mathbf{A}_{i, k} \mathbf{B}_{k, j}=\mathrm{R}_{\sigma_{2}}\left(\Phi_{i, k, j }\right)(\operatorname{vec}(\mathbf{A}), \operatorname{vec}(\mathbf{B}))
		\end{equation}
		and  $L\left(\Phi_{i, k, j }\right)=L\left(\Phi\right)$. By Lemma \ref{1},  $\Phi_{i, k, j }$  also satisfies \eqref{4}-\eqref{6} with $\Phi$ replaced by $\Phi_{i,j,k}$. 
		
		Next, we define
		\begin{equation*}
			\Phi_{i, j}:=\left(\left(\mathbf{1}_{\mathbb{R}^{n}}, 0\right)\right) \bullet \mathrm{P}\left(\Phi_{i, 1, j }, \ldots, \Phi_{i, n, j }\right) \bullet\left(\left(\left(\begin{array}{c}
				\mathbf{I} \mathbf{d}_{\mathbb{R}^{n(d+l)}} \\
				\vdots \\
				\mathbf{I} \mathbf{d}_{\mathbb{R}^{n(d+l)}}
			\end{array}\right), \mathbf{0}_{\mathbb{R}^{n^{2}(d+l)}}\right)\right),
		\end{equation*}
		where $\mathbf{1}_{\mathbb{R}^{n}} \in \mathbb{R}^{n}$ is a vector with each entry equal to 1. Next, it follows that $\Phi_{i, j }$ is a neural network with $n \cdot(d+l)$ dimensional input and one-dimensional output, satisfying the following expression:
		\begin{equation}\label{ab}
			\mathrm{R}_{\sigma_{2}}\left(\Phi_{i, j }\right)(\operatorname{vec}(\mathbf{A}), \operatorname{vec}(\mathbf{B}))=\sum_{k=1}^{n} \mathrm{R}_{\sigma_{2}}\left(\Phi_{i, k, j }\right)(\operatorname{vec}(\mathbf{A}), \operatorname{vec}(\mathbf{B}))=\sum_{k=1}^{n} \mathbf{A}_{i, k} \mathbf{B}_{k, j} .
		\end{equation}
		and \eqref{3} with $\Phi$ replaced by $\Phi_{i, j}$. Moreover, by Lemma \ref{1}, Lemma \ref{11b},  and \eqref{4}-\eqref{6}, we obtain the following expression:
		\begin{equation}\label{33}
			M\left(\Phi_{i, j }\right) \leq M\left(\mathrm{P}\left(\Phi_{i, 1, j} , \ldots, \Phi_{i, n, j }\right)\right) \leq 12n,
		\end{equation}
		\begin{equation}\label{55}
			M_{1}\left(\Phi_{i, j }\right) \leq M_{1}\left(\mathrm{P}\left(\Phi_{i, 1, j }, \ldots, \Phi_{i, n, j }\right)\right) \leq 8n,
		\end{equation}
		\begin{equation}\label{54}
			M_{L\left(\Phi_{i, j }\right)}\left(\Phi_{i, j }\right) \leq M_{L\left(\Phi_{i, j }\right)}\left(\mathrm{P}\left(\Phi_{i, 1, j }, \ldots, \Phi_{i, n, j }\right)\right) \leq 4 n.
		\end{equation}
		Finally, we define the ReQU neural network $\Phi_{\mathrm{mult} }^{d, n, l}$ with $n\cdot(d+l)$ dimensional input and $dl$ dimensional output as follows:
		\begin{equation}\label{234}
			\Phi_{\mathrm{mult}}^{d, n, l}:=\mathrm{P}\left(\Phi_{1,1 }, \ldots, \Phi_{d, 1 }, \ldots, \Phi_{1, l }, \ldots, \Phi_{d, l }\right) \bullet\left(\left(\left(\begin{array}{c}
				\mathbf{Id}_{\mathbb{R}^{n(d+l)}} \\
				\vdots \\
				\mathbf{Id}_{\mathbb{R}^{n(d+l)}}
			\end{array}\right), \mathbf{0}_{\mathbb{R}^{d l n(d+l)}}\right)\right).
		\end{equation}
		Then, by \eqref{ab}, we have the following expression:
		\begin{equation*}
			\operatorname{matr}\left(\mathrm{R}_{\sigma_{2}}\left(\Phi_{\mathrm{mult}}^{d, n, l}\right)(\operatorname{vec}(\mathbf{A}), \operatorname{vec}(\mathbf{B}))\right) = \mathbf{A B}.
		\end{equation*}
		By Lemma \ref{11b},  \eqref{3} is satisfied with $\Phi$ replaced by $\Phi_{\mathrm{mult}}^{d, n, l}$. Then, by Lemma \ref{1},  Lemma \ref{11b}, and \eqref{33}, we have the following expression:
		\begin{equation*}
			M\left(\Phi_{\mathrm{mult}}^{d, n, l}\right) \leq 12 d n l.
		\end{equation*}
		Furthermore, by Lemma \ref{11b}, \eqref{55}, and \eqref{54}, we derive
		$$
		M_{1}\left(\Phi_{\mathrm{mult}}^{d, n, l}\right) \leq 8 d n l \text { and } M_{L\left(\Phi_{\mathrm{mult}}^{d, n, l}\right)}\left(\Phi_{\mathrm{mult}}^{d, n, l}\right) \leq 4 d n l.
		$$
		Thus, we finish the proof.
	\end{proof}
	
	In particular, let $d = n = l$; then, we have the following corollary.
	\begin{cor}\label{2.3}
		Let $d \in \mathbb{N}$.   There exists a  ReQU neural network $\Phi_{\mathrm{mult}}^{d, d, d}$ with $2d^2$ dimensional input, $d^2$ dimensional output such that the following properties are satisfied:
		\begin{enumerate}[(i)]
			\item$\operatorname{matr}\left(\mathrm{R}_{\sigma_{2}}\left(\Phi_{\mathrm{mult}}^{d, d, d}\right)(\operatorname{vec}(\mathbf{A}), \operatorname{vec}(\mathbf{B}))\right) = \mathbf{A B}$, for any $\mathbf{A}, \mathbf{B}\in\mathbb{R}^{d\times d }$,
			\item $L\left(\Phi_{\mathrm{mult}}^{d, d, d}\right)=2$,
			\item $M\left(\Phi_{\mathrm{mult}}^{d, d, d}\right) \leq 12 d^{3}$,
			\item $M_{1}\left(\Phi_{\mathrm{mult}}^{d, d, d}\right) \leq 8 d^{3}, \quad$ as well as $M_{L\left(\Phi_{\mathrm{mult}}^{d, d, d}\right)}\left(\Phi_{\mathrm{mult}}^{d, d, d}\right) \leq 4 d^{3}$.
		\end{enumerate}			
	\end{cor}
	Therefore, we can simulate matrix squaring as follows: 
	\begin{defn}\label{d}
		For $d \in \mathbb{N}$, we define the ReQU neural network
		$$
		\Phi_{2 }^{d}:=\Phi_{\mathrm{mult}}^{d, d, d} \bullet\left(\left(\left(\begin{array}{l}
			\mathbf{I} \mathbf{d}_{\mathbb{R}^{d^{2}}} \\
			\mathbf{I} \mathbf{d}_{\mathbb{R}^{d^{2}}}
		\end{array}\right), \mathbf{0}_{\mathbb{R}^{2 d^{2}}}\right)\right),
		$$
		which has $d^{2}$ dimensional input and $d^{2}$ dimensional output.
	\end{defn}
	Then, according to Proposition \ref{p}, for all $d \in \mathbb{N}$, we have
	\begin{prop} The ReQU neural network $\Phi_{2 }^{d}$ satisfies the following properties:
		\begin{enumerate}[(i)]
			\item$\operatorname{matr}\left(\mathrm{R}_{\sigma_{2}}\left(\Phi_{2 }^{d}\right)(\operatorname{vec}(\mathbf{A}))\right) = \mathbf{A}^{2},$  for any $\mathbf{A}\in\mathbb{R}^{d\times d }$,
			\item $L\left(\Phi_{2 }^{d}\right) =2$,
			\item$M\left(\Phi_{2 }^{d}\right) \leq 12 d^{3}, $
			\item$M_{1}\left(\Phi_{2 }^{d}\right) \leq 8 d^{3}, \quad$ as well as $\quad M_{L\left(\Phi_{2 }^{d}\right)}\left(\Phi_{2 }^{d}\right) \leq 4 d^{3}$.
		\end{enumerate}
	\end{prop}
	
	\subsection{ReQU Neural Network to Represent the Power Function of the Matrix}
	Based on the aforementioned results, we can represent map $\mathbf{A} \mapsto \mathbf{A}^{2^j}$ using a ReQU neural network, for arbitrary $j \in \mathbb{N}$.
	\begin{prop}\label{345}
		Let $d \in \mathbb{N}, j \in \mathbb{N}.$ There exists a ReQU neural network $ \Phi_{2^{j} }^{d}$ with $d^{2}$ dimensional input and output that satisfies the following properties:
		\begin{enumerate}[(i)]
			\item$\operatorname{matr}\left(\mathrm{R}_{\sigma_{2}}\left(\Phi_{2^{j}}^{d}\right)(\operatorname{vec}(\mathbf{A}))\right)=\mathbf{A}^{2^{j}},$ for any $\mathbf{A}\in\mathbb{R}^{d\times d }$,
			\item $L\left(\Phi_{2^{j} }^{d}\right) = 2 j,$
			\item $M\left(\Phi_{2^{j} }^{d}\right) \leq 64j d^{3},$
			\item $M_{1}\left(\Phi_{2^{j} }^{d}\right) \leq 8 d^{3}, \quad$ as well as $\quad M_{L\left(\Phi^{d}_{2^ j}\right)}\left(\Phi_{2^{j} }^{d}\right) \leq 4 d^{3}.$
		\end{enumerate}

	\end{prop}
	\begin{proof}
		We prove the statement by induction. For $j=1$, the statement proceeds by selecting $\Phi_{2 }^{d}$ as in Definition \ref{d}. Next, we assume that the statement holds for an arbitrary, but fixed $j \in \mathbb{N}$; that is, there exists a neural network $\Phi_{2^{j} }^{d}$ satisfying (i)-(iv). Next, we define the following expression:
		\begin{equation}\label{2^j}
			\Phi_{2^{j+1}}^{d}:=\Phi_{2 }^{d} \odot \Phi_{2^{j}}^{d}.
		\end{equation}
		By inductive hypothesis,
		\begin{equation*}
			\operatorname{matr}\left(\mathrm{R}_{\sigma_{2}}\left(\Phi_{2^{j+1} }^{d}\right)(\operatorname{vec}(\mathbf{A}))\right)=\left(\operatorname{matr}\left(\mathrm{R}_{\sigma_{2}}\left(\Phi_{2^{j} }^{d}\right)(\operatorname{vec}(\mathbf{A}))\right)\right)^{2}=\mathbf{A}^{2^{j+1}}.
		\end{equation*}
		We estimate the size of $\Phi_{2^{j+1} }^{d}$. By Lemma \ref{11}, we obtain the following expression:
		\begin{equation*}
			\begin{aligned}
				L\left(\Phi_{2^{j+1} }^{d}\right)
				&=L\left(\Phi_{2}^{d}\right)+L\left(\Phi_{2^{j} }^{d}\right)= 2+2j =2(j+1),
			\end{aligned}
		\end{equation*}
		and
		\begin{equation*}
			\begin{aligned}
				M\left(\Phi_{2^{j+1}}^{d}\right) &\leq M\left(\Phi_{2}^{d}\right)+M\left(\Phi_{2^{j} }^{d}\right)+4M_{1}\left(\Phi_{2 }^{d}\right)+4M_{L\left(\Phi_{2 ^ j }^{d}\right)}\left(\Phi_{2^{j} }^{d}\right)+4d^2\\
				&\leq 12d^3+64j d^3 + 4(4d^3 + 8d^3)+4d^2\\
				&\leq 64(j+1)d^3.
			\end{aligned}
		\end{equation*}
		Moreover,
		\begin{equation*}
			M_{1}\left(\Phi_{2^{j+1} }^{d}\right)=M_{1}\left(\Phi_{2^{j} }^{d}\right) \leq 8 d^{3},
		\end{equation*}
		\begin{equation*}
			M_{L\left(\Phi_{2^{j+1}}^{ d}\right)}\left(\Phi_{2^{j+1} }^{ d}\right)=M_{L\left(\Phi_{2 }^{d}\right)}\left(\Phi_{2 }^{d}\right) \leq 4 d^{3}.
		\end{equation*}
		The proof is completed.
	\end{proof}

	\subsection{Proof of Theorem \ref{3.4}}
	Next, we construct a neural network $\Phi_{\mathrm {inv}; \epsilon}^{d}$ that approximates the inversion operator, that is, the map $\mathbf{A} \mapsto \left(\mathbf{I} \mathbf{d}_{\mathbb{R}^{d}}-\mathbf{A}\right)^{-1}$ up to accuracy $\epsilon>0$. Based on the properties of the partial sums of the \textit{Neumann series}, for every $\epsilon,\delta\in(0,1)$ and $\mathbf{A}\in\mathbb{R}^{d\times d}$ satisfying $\|\mathbf{A}\|_2\leq 1-\delta$, we have the following expression:
	\begin{equation*}
		\begin{aligned}
			\left\|\left(\mathbf{I} \mathbf{d}_{\mathbb{R}^{d}}-\mathbf{A}\right)^{-1}-\sum_{k=0}^{2^l-1} \mathbf{A}^{k}\right\|_{2} &=\left\|\left(\mathbf{I} \mathbf{d}_{\mathbb{R}^{d}}-\mathbf{A}\right)^{-1} \mathbf{A}^{2^l}\right\|_{2} \leq\left\|\left(\mathbf{Id}_{\mathbb{R}^{d}}-\mathbf{A}\right)^{-1}\right\|_{2}\|\mathbf{A}\|_{2}^{2^l} \\
			& \leq \frac{1}{1-(1-\delta)} \cdot(1-\delta)^{2^l}=\frac{(1-\delta)^{2^l}}{\delta}\leq \epsilon
		\end{aligned}.
	\end{equation*}
	Here, we have used the fact
	\begin{equation*}
		l(\epsilon ,\delta )=\left\lceil\log_{2}(\log_{1-\delta}(\delta\epsilon)+1)\right\rceil\geq \log_{2}(\log_{1-\delta}(\delta\epsilon)+1).
	\end{equation*}
	Thus, it suffices to construct a ReQU neural network representing $\sum_{k=0}^{2^l-1} \mathbf{A}^{k}$.
	Note that
	\begin{equation*}
		\sum_{k=0}^{2^l-1} \mathbf{A}^{k}=\prod_{i=0}^{l-1}(\mathbf{A}^{2^i}+I);
	\end{equation*}
	by Proposition \ref{345} and Corollary \ref{2.3}, we have ReQU neural networks $ \Phi_{2^{j} }^{d} $ and $\Phi_{\mathrm{mult}}^{d, d, d}$ such that
	\begin{equation*}
		\operatorname{matr}\left(\mathrm{R}_{\sigma_{2}}\left(\Phi_{2^{j}}^{d}\right)(\operatorname{vec}(\mathbf{A}))\right) = \mathbf{A}^{2^{j}},\quad \forall j\in \mathbb{N},
	\end{equation*}
	\begin{equation*}
		\operatorname{matr}\left(R_{\sigma_{2}}\left(\Phi_{\mathrm{mult}}^{d, d, d}\right)(\operatorname{vec}(\mathbf{A}), \operatorname{vec}(\mathbf{B}))\right) = \mathbf{A B}.
	\end{equation*}
	Let  $ \textbf{b} =\textbf{vec}(\mathbf{I} \mathbf{d}_{\mathbb{R}^{d^{2}}})  $ and define $ \Phi_{1}:=((\mathbf{I} \mathbf{d}_{\mathbb{R}^{d^2}},\textbf{b})) $ and
	
	\begin{equation*}
		\pi_{l}:=\left\{\begin{array}{ll}
			\Phi_{1}, & \text{if}~l = 1,  \\
			\Phi_{\mathrm{mult}}^{d, d, d}\odot P(\pi_{l-1},\Phi_{1}\odot\Phi_{2^{l-1}}^{d}), & \text{if}~l \geq2.
		\end{array}\right.
	\end{equation*}
	Next, we set
	
	\begin{equation*}
		\Phi_{\mathrm{inv} }^{d}:=\pi_{l}\bullet \left(\left(\left(\begin{array}{c}
			\mathbf{I d}_{\mathbb{R}^{d^{2}}} \\
			\vdots \\
			\mathbf{I d}_{\mathbb{R}^{d^{2}}}
		\end{array}\right), \mathbf{0}_{\mathbb{R}^{l d^{2}}}\right)\right),
	\end{equation*}
	which is a ReQU neural network with $d^2$ dimensional input and output.
	By the definition of $\pi_{l}$, for every $\mathbf{A}\in\mathbb{R}^{d\times d}$ such that $\|\mathbf{A}\|_2\leq 1-\delta$, we have
	$$\operatorname{matr}\left(\mathrm{R}_{\sigma_{2}}\left(\Phi_{\mathrm{inv} }^{d}\right)(\operatorname{vec}(\mathbf{A}))\right) = \sum_{k=0}^{2^l-1} \mathrm{A}^{k},$$
	which implies (i) of Theorem \ref{3.4}.
	
	Next, we analyze the size of the resulting neural network $\Phi_{\mathrm{inv} }^{d}$. First, by Lemma \ref{11}, Lemma \ref{11b}, and Proposition \ref{345}, we have the following expression:
	\begin{align*}
		\quad L\left(\Phi_{\mathrm{inv} }^{d}\right)= L(\pi_{l})=L(\Phi_{\mathrm{mult}}^{d, d, d})+\max\{L(\pi_{l-1}),L(\Phi_{1} \odot \Phi_{2^{l-1}}^{d})\}=2+\max\{L(\pi_{l-1}),2l-1\}
	\end{align*}
	Note that $L(\pi_{2}) = 5$, by induction, we have
	$$L(\pi_{l}) = 2l+1.$$
	For the nonzero weights, we have
	\begin{equation}\label{weight}
		\begin{aligned} M\left(\Phi_{\mathrm{inv} }^{d}\right)&\leq M(\pi_{l})=M(\Phi_{\mathrm{mult}}^{d, d, d}\odot P(\pi_{l-1},\Phi_{1}\odot\Phi_{2^{l-1}}^{d}))\\
			&\leq M(\Phi_{\mathrm{mult}}^{d, d, d})+M(P(\pi_{l-1},\Phi_{1}\odot\Phi_{2^{l-1}}^{d}))\\
			&\quad+4M_1(\Phi_{\mathrm{mult}}^{d, d, d})+4M_{L(P(\pi_{l-1},\Phi_{1}\odot\Phi_{2^{l-1}}^{d}))}(P(\pi_{l-1},\Phi_{1}\odot\Phi_{2^{l-1}}^{d}))+4d^2\\
			&\leq 12d^3+M(P(\pi_{l-1},\Phi_{1}\odot\Phi_{2^{l-1}}^{d}))+4\times 8d^3\\
			&\quad+4M_{L(P(\pi_{l-1},\Phi_{1}\odot\Phi_{2^{l-1}}^{d}))}(P(\pi_{l-1},\Phi_{1}\odot\Phi_{2^{l-1}}^{d}))+4d^2\\
			&\leq 44d^3+4d^2+M(P(\pi_{l-1},\Phi_{1}\odot\Phi_{2^{l-1}}^{d}))\\
			&\quad +4 \max\{4\mathrm{dim}_{\mathrm{out}}(\pi_{l-1}),M_{L(\pi_{l-1})}(\pi_{l-1})\}\\
			&\quad+ 4 \max\{4\mathrm{dim}_{\mathrm{out}}(\Phi_{1}\odot\Phi^{d}_{2 ^ {l-1} }),M_{L(\Phi_{1}\odot\Phi^{d}_{2 ^ {l-1} })}(\Phi_{1}\odot\Phi^{d}_{2 ^ {l-1} })\}.\\
		\end{aligned}
	\end{equation}
	Here, by \eqref{3.9} and the definition of $\Phi_{1}$, we have
	\begin{align*}
		M_{L(\Phi_{1}\odot\Phi^{d}_{2 ^ {l-1} })}(\Phi_{1}\odot\Phi^{d}_{2 ^ {l-1} }) = 4 \| \mathbf{Id}_{\mathbb{R}^{d^2}}\|_0 +\|\mathbf{b}\|_0 = 4d^2 + d^2 =5 d^2,
	\end{align*}
	and
	\begin{align}\label{weight3}
		\max\{4\mathrm{dim}_{\mathrm{out}}(\Phi_{1}\odot\Phi^{d}_{2 ^ {l-1} }),M_{L(\Phi_{1}\odot\Phi^{d}_{2 ^ {l-1} })}(\Phi_{1}\odot\Phi^{d}_{2 ^ {l-1} })\} = \max\{4d^2,5d^2\} = 5d^2.
	\end{align}
	By the definition of $\pi_{l-1}$, Lemma \ref{11}, and Corollary \ref{2.3}, we have
	\begin{align}\label{weight2}
		\max\{4\mathrm{dim}_{\mathrm{out}}(\pi_{l-1}),M_{L(\pi_{l-1})}(\pi_{l-1})\} = \max\{4d^2,M_{L\left(\Phi_{\mathrm{mult}}^{d, d, d}\right)}\left(\Phi_{\mathrm{mult}}^{d, d, d}\right) \} \leq 4d^3.
	\end{align}
	Moreover, 
	\begin{equation}\label{weight1}
		\begin{aligned}
			&M(P(\pi_{l-1},\Phi_{1}\odot\Phi_{2^{l-1}}^{d}))\\
			\leq &M(\pi_{l-1}) + M(\Phi_{1}\odot\Phi_{2^{l-1}}^{d}) + 4M_{L(\pi_{l-1})}(\pi_{l-1})+ 4M_{L(\Phi_{1}\odot\Phi_{2^{l-1}}^{Z, d})}(\Phi_{1}\odot\Phi_{2^{l-1}}^{d})\\
			& +\mathrm{dim}_{\text {out }}(\pi_{l-1})(20L(\pi_{l-1})+8)+\mathrm{dim}_{\text {out}}(\Phi_{1}\odot\Phi_{2^{l-1}}^{d})(20L(\Phi_{1}\odot\Phi_{2^{l-1}}^{d})+8)\\
			\leq& M(\pi_{l-1})+(2d^2+64(l-1)d^3+4\times 2d^2+4\time 4\times 4d^3+4d^2)+4(4d^3+5d^2)\\
			&+80ld^2-24d^2\\
			=& M(\pi_{l-1}) +64ld^3 -32d^3 +80ld^2  + 10d^2.
		\end{aligned}
	\end{equation}
	
	By substituting \eqref{weight3}-\eqref{weight1} into \eqref{weight}, we have 
	\begin{equation*}
		\begin{aligned}
			M(\pi_{l})&\leq(64l+28)d^3+(80l+34)d^2+M(\pi_{l-1})\\
			&\leq \sum_{j=2}^{l}((64j+28)d^3+(80j+34)d^2)+M(\pi_{1})\\
			&\leq \left [64\frac{(2+l)(l-1)}{2} +28(l-1)\right ]d^3+\left [ 80\frac{(2+l)(l-1)}{2}+34(l-1)\right ]d^2+2d^2\\
			&=(32l^2+60l-80)d^3+(40l^2-44l-112)d^2\\
			&\leq C_{\mathrm{inv}}d^3l^2,
		\end{aligned}
	\end{equation*}
	where $C_{\mathrm{inv}}$ is a universal constant. The proof is completed.

	\section{Theoretical Analysis of the ReQU Neural Network to Solve Parametric PDE}
	In this section, we use ReQU neural networks to approximate the parameter-dependent solution of parametric PDEs, and derive the complexity bounds. As mentioned previously, for any $\tilde{\epsilon} \geq \hat{\epsilon}>0$, there exists reduced basis space $H_{\tilde{\epsilon}}^{\mathrm{rb}} = \operatorname{span}\{\psi_i\}_{i=1}^{d(\tilde{\epsilon})}$ and corresponding reduced basis solution
	$$u_{y, \tilde{\epsilon}}^{\mathrm{rb}}=\sum_{i=1}^{d(\tilde{\epsilon})}\left(\mathbf{u}_{y, \tilde{\epsilon}}^{\mathrm{rb}}\right)_{i} \psi_{i}=\sum_{j=1}^{D}\left(\tilde{\mathbf{u}}^{\mathrm{h}}_{y, \tilde{\epsilon}}\right)_{j} e_{j},$$
	such that
	\begin{equation*}
		\sup _{y \in \mathcal{Y}}\left\|u_{y}-u_{y, \tilde{\epsilon}}^{\mathrm{rb}}\right\|_{H} \leq \frac{\alpha}{\beta} \tilde{\epsilon},
	\end{equation*}
	Thus, it suffices to construct ReQU neural networks approximating the maps 
	\begin{equation*}
		\mathcal{Y} \rightarrow \mathbb{R}^{d(\tilde{\epsilon})}: \quad y \mapsto \mathbf{u}_{y, \tilde{\epsilon}}^{\mathrm{rb}}, \text { and }\mathcal{Y} \rightarrow \mathbb{R}^{m}: \quad y \mapsto \tilde{\mathbf{u}}^{\mathrm{h}}_{y, \tilde{\epsilon}} .
	\end{equation*}
	Based on the discussion in Section \ref{section2}, we have the following expression:
	\begin{equation*}
		\mathbf{u}_{y, \tilde{\epsilon}}^{\mathrm{rb}}:=\left(\mathbf{B}_{y, \tilde{\epsilon}}^{\mathrm{rb}}\right)^{-1} \mathbf{f}_{y, \tilde{\epsilon}}^{\mathrm{rb}},
	\end{equation*}
	where
	\begin{equation*}
		\mathbf{B}_{y, \tilde{\epsilon}}^{\mathrm{rb}}:=\left(B_{y}\left(\psi_{j}, \psi_{i}\right)\right)_{i, j=1}^{d(\tilde{\epsilon})}, \quad \mathbf{f}_{y, \tilde{\epsilon}}^{\mathrm{rb}}:=\left(f_{y}\left(\psi_{i}\right)\right)_{i=1}^{d(\tilde{\varepsilon})} \quad \text { for all } y \in \mathcal{Y}.
	\end{equation*}
	
	Our strategy is to first approximate $(\mathbf{B}_{y, \tilde{\epsilon}}^{\mathrm{rb}})^{-1}$, $\mathbf{f}_{y, \tilde{\epsilon}}^{\mathrm{rb}}$, and then, approximate the multiplication of them. By Theorem \ref{3.4}, we can approximate $(\mathbf{B}_{y, \tilde{\epsilon}}^{\mathrm{rb}})^{-1}$ by applying the neural network $\Phi_{\mathrm{inv} ; \epsilon}^{d}$ to the matrix $\mathbf{I d}_{\mathbb{R}^{d(\tilde{\epsilon})}}-\lambda \mathbf{B}_{y, \tilde{\epsilon}}^{\mathrm{rb}}$. Here, $\lambda$ is the scaling factor to ensure $\left\|\mathbf{I d}_{\mathbb{R}^{d(\tilde{\epsilon})}}-\lambda \mathbf{B}_{y, \tilde{\epsilon}}^{\mathrm{rb}}\right\|_{2} <1$. We can select $\lambda:=\left(\alpha + \beta\right)^{-1}$ (independent of $y$ and $d(\tilde{\epsilon}))$ such that with $\delta:=\lambda \beta$
	$$
	\left\|\mathbf{I} \mathbf{d}_{\mathbb{R}^{d(\tilde{\varepsilon})}}-\lambda \mathbf{B}_{y, \tilde{\epsilon}}^{\mathrm{rb}}\right\|_{2}=\max _{\mu \in \sigma\left(\mathbf{B}_{y, \tilde{\varepsilon}}^{\mathrm{rb}}\right)}|1-\lambda \mu| \leq \max _{\mu \in\left[\beta, \alpha\right]}|1-\lambda \mu|=1-\lambda \beta=1-\delta<1
	$$
	for all $y \in \mathcal{Y}, \tilde{\epsilon}>0$.
	We fix these values of $\lambda$ and $\delta$ for the remainder of this paper. To proceed, we state the assumptions on the approximability of the map $\mathbf{B}_{\cdot, \tilde{\epsilon}}^{\mathrm{rb}}$ and $\mathbf{f}_{\cdot, \tilde{\epsilon}}^{\mathrm{rb}}$.
	\begin{assu}\label{4.1}
		For any $\tilde{\epsilon} \geq \hat{\epsilon}, \epsilon>0$, and a corresponding reduced basis $\left(\psi_{i}\right)_{i=1}^{d(\tilde{\epsilon})}$, there exists a neural network $ \Phi_{\tilde{\epsilon}, \epsilon}^{\mathbf{B}}$ with $p$ dimensional input and $d(\tilde{\epsilon})^{2}$ dimensional output such that
		$$
		\sup _{y \in \mathcal{Y}}\left\|\lambda \mathbf{B}_{y, \tilde{\epsilon}}^{\mathrm{rb}}-\operatorname{matr}\left(\mathrm{R}_{\sigma_{2}}\left(\Phi_{\tilde{\epsilon}, \epsilon}^{\mathbf{B}}\right)(y)\right)\right\|_{2} \leq \epsilon.
		$$
		We set $B_{L}(\tilde{\epsilon}, \epsilon):=L\left(\Phi_{\tilde{\epsilon}, \epsilon}^{\mathbf{B}}\right)$
		and 
		$B_{M}(\tilde{\epsilon}, \epsilon):=M\left(\Phi_{\tilde{\epsilon}, \epsilon}^{\mathbf{B}}\right) $.
	\end{assu}
	\begin{assu}\label{4.2}
		For every $\tilde{\epsilon} \geq \hat{\epsilon}, \epsilon>0$, and a corresponding reduced basis $\left(\psi_{i}\right)_{i=1}^{d(\tilde{\epsilon})}$, there exists a neural network $\Phi_{\tilde{\epsilon}, \epsilon}^{\mathbf{f}}$ with $p$ dimensional input and $d(\tilde{\epsilon})$ dimensional output such that
		$$
		\sup _{y \in \mathcal{Y}}\left|\mathbf{f}_{y, \tilde{\epsilon}}^{\mathrm{rb}}-\mathrm{R}_{\sigma_{2}}\left(\Phi_{\tilde{\epsilon}, \epsilon}^{\mathbf{f}}\right)(y)\right| \leq \epsilon.
		$$
		We set $F_{L}(\tilde{\epsilon}, \epsilon):=L\left(\Phi_{\tilde{\epsilon}, \epsilon}^{\mathbf{f}}\right)$ and $F_{M}(\tilde{\epsilon}, \epsilon):=M\left(\Phi_{\tilde{\epsilon}, \epsilon}^{\mathbf{f}}\right).$
	\end{assu}
	
	Next, we present the construction of the neural network emulating $y \mapsto\left(\mathbf{B}_{y, \tilde{\epsilon}}^{\mathrm{rb}}\right)^{-1}$.
	\begin{prop}\label{4.7}
		Let $\tilde{\epsilon} \geq \hat{\epsilon}, \epsilon >0,$ and $\epsilon^{\prime}:=\min\{\frac{3 \epsilon \lambda \beta^{2}}{8}, \frac{\lambda\beta}{4}\}.$ Under  Assumption
		\ref{4.1}, there exists a ReQU neural network $\Phi_{\mathrm{inv} ;\tilde{\epsilon},\epsilon}^{\textbf{B}}$ with $p$ dimensional input and $d(\tilde{\epsilon})^{2}$ dimensional output that satisfies the following properties:
		\begin{enumerate}[(i)]
			\item $\sup _{y \in \mathcal{Y}}\left\|\left(\mathbf{B}_{y, \tilde{\epsilon}}^{\mathrm{rb}}\right)^{-1}-\operatorname{matr}\left(\mathrm{R}_{\sigma_{2}}\left(\Phi_{\mathrm{inv} ; \tilde{\epsilon}, \epsilon}^{\textbf{B}}\right)(y)\right)\right\|_{2} \leq \epsilon,$
			\item there exists a constant $C^{B}_{L}=C^{B}_{L}(\alpha,\beta)>0$  such that 
			$$L\left(\Phi_{\mathrm{inv} ; \tilde{\epsilon}, \epsilon}^		\textbf{B}\right) \leq C_{B} \log _{2}(\log_{2}(1/\epsilon))+B_{L}(\tilde{\epsilon},\epsilon^{\prime}),$$
			\item there exists a constant $C^{B}_{M}=C^{B}_{M}(\alpha,\beta)>0$  such that 
			$$M\left(\Phi_{\mathrm{inv}; \tilde{\epsilon}, \epsilon}^{\textbf{B}}\right) \leq C_B d^3(\tilde{\epsilon})\log^{2}_{2}(\log_{2}(1/\epsilon))+5B_M(\tilde{\epsilon},\epsilon^{\prime}).$$ 
		\end{enumerate}
	\end{prop}

	\begin{proof}
		Let $\left(\left(\mathbf{A}_{\tilde{\epsilon}, \epsilon'}^{1}, \mathbf{b}_{\tilde{\epsilon}, \epsilon'}^{1}\right), \ldots,\left(\mathbf{A}_{\tilde{\epsilon}, \epsilon'}^{L}, \mathbf{b}_{\tilde{\epsilon}, \epsilon'}^{L}\right)\right):=\Phi_{\tilde{\epsilon}, \epsilon'}^{\mathbf{B}}$ be the neural network in Assumption \ref{4.1}, then, for
		$$
		\Phi_{\tilde{\epsilon}, \epsilon'}^{\mathbf{B}, \mathbf{I d}}:=\left(\left(\mathbf{A}_{\tilde{\epsilon}, \epsilon'}^{1}, \mathbf{b}_{\tilde{\epsilon}, \epsilon'}^{1}\right), \ldots,\left(-\mathbf{A}_{\tilde{\epsilon}, \epsilon'}^{L}, -\mathbf{b}_{\tilde{\epsilon}, \epsilon'}^{L}+\operatorname{vec}\left(\mathbf{I} \mathbf{d}_{\mathbb{R}^{d(\tilde{\epsilon})}}\right)\right)\right),
		$$
		we have 
		\begin{align}\label{eq-esti-Phi-BId}
			\sup _{y \in \mathcal{Y}}\left\|\mathbf{I d}_{\mathbb{R}^{d(\tilde{\epsilon})}}-\lambda \mathbf{B}_{y, \tilde{\epsilon}}^{\mathrm{rb}}-\operatorname{matr}\left(\mathrm{R}_{\sigma_{2}}\left(\Phi_{\tilde{\epsilon}, \epsilon'}^{\mathbf{B}, \mathbf{I d}}\right)(y)\right)\right\|_{2} \leq \epsilon',
		\end{align}
		as well as $M\left(\Phi_{\tilde{\epsilon}, \epsilon'}^{\mathbf{B}, \mathbf{I d}}\right) \leq B_{M}(\tilde{\epsilon}, \epsilon')+d^{2}(\tilde{\epsilon})$ and $L\left(\Phi_{\tilde{\epsilon}, \epsilon'}^{\mathbf{B}, \mathbf{I d}}\right)=B_{L}(\tilde{\epsilon}, \epsilon')$.
		Next, we define the ReQU neural network
		$$
		\Phi_{\mathrm{inv} ; \tilde{\epsilon}, \epsilon}^{\textbf{B}}:=\left(\left(\lambda \mathbf{I} \mathbf{d}_{\mathbb{R}^{d(\tilde{\epsilon})}}, \mathbf{0}_{\mathbb{R}^{d(\tilde{\epsilon})}}\right)\right) \bullet \Phi_{\mathrm{inv} ; \frac{\epsilon}{2 \lambda}}^{ d(\tilde{\epsilon})} \odot \Phi_{\tilde{\epsilon}, \epsilon^{\prime}}^{\mathbf{B}, \mathbf{I d}},
		$$
		with $p$ dimensional input and $d(\tilde{\epsilon})^{2}$ dimensional output. To prove $(i)$, it suffices to estimate
		\begin{equation}\label{4.5}
			\begin{aligned}
				&\left\| \left( \lambda\mathbf{B}_{y, \tilde{\epsilon}}^{\mathrm{rb}}\right)^{-1} - \operatorname{matr} \left( \mathrm{R}_{\sigma_{2}} \left( \Phi_{\mathrm{inv} ; \frac{\epsilon}{2 \lambda}}^{d(\tilde{\epsilon})} \odot \Phi_{\tilde{\epsilon}, \epsilon^{\prime}}^{\textbf{B}, \textbf{Id}} \right) (y) \right)\right\|_{2} \\
				\leq &  \left\|\left(\lambda\mathbf{B}_{y, \tilde{\epsilon}}^{\mathrm{rb}}\right)^{-1}-\left( \mathbf{I d}_{\mathbb{R}^{d(\tilde{\epsilon})}} - \operatorname{matr} \left( \mathrm{R}_{\sigma_{2}} \left(\Phi_{\tilde{\epsilon}, \epsilon^{\prime}}^{\textbf{B}, \textbf{Id}}\right)(y)\right)\right)^{-1}\right\|_{2} \\
				&+\left\| \left( \mathbf{I d}_{\mathbb{R}^{d(\tilde{\epsilon})}} - \operatorname{matr} \left( \mathrm{R}_{\sigma_{2}} \left(\Phi_{\tilde{\epsilon}, \epsilon^{\prime}}^{\textbf{B}, \textbf{Id}}\right)(y)\right)\right)^{-1} - \operatorname{matr}\left(\mathrm{R}_{\sigma_{2}}\left(\Phi_{\mathrm{inv} ; \frac{\epsilon}{2 \lambda}}^{d(\tilde{\epsilon})} \odot \Phi_{\tilde{\epsilon}, \epsilon^{\prime}}^{\textbf{B}, \textbf{Id}}\right)(y)\right)\right\|_{2}\\
				=:& \mathrm{I}+\mathrm{II}.
			\end{aligned}
		\end{equation}
		
		For term $\mathrm{I}$, by \eqref{eq-esti-Phi-BId}, we have the following expression:
		\begin{align*}
			\left\|\mathbf{I d}_{\mathbb{R}^{d(\tilde{\epsilon})}} - \operatorname{matr} \left( \mathrm{R}_{\sigma_{2}} \left(\Phi_{\tilde{\epsilon}, \epsilon^{\prime}}^{\textbf{B}, \textbf{Id}}\right)(y)\right)\right\|_2 \ge \|\lambda \mathbf{B}_{y, \tilde{\epsilon}}^{\mathrm{rb}}\|_2 -\epsilon' \ge \lambda \beta - \frac{\lambda \beta}{4} =  \frac{3}{4}\lambda \beta.
		\end{align*}
		By combining Assumption \ref{4.1}, \eqref{B}, and \eqref{4.5}, we obtain
		\begin{equation*}
			\begin{aligned}
				\mathrm{I} & \leq\left\|\left(\mathbf{I d}_{\mathbb{R}^{d(\tilde{\epsilon})}} - \operatorname{matr} \left( \mathrm{R}_{\sigma_{2}} \left(\Phi_{\tilde{\epsilon}, \epsilon^{\prime}}^{\textbf{B}, \textbf{Id}}\right)(y)\right)\right)^{-1}\right\|_{2} \left\|\left(\lambda \mathbf{B}_{y, \tilde{\epsilon}}^{\mathrm{rb}}\right)^{-1}\right\|_{2} \left\|\mathbf{I d}_{\mathbb{R}^{d(\tilde{\epsilon})}}-\lambda \mathbf{B}_{y, \tilde{\epsilon}}^{\mathrm{rb}}-\operatorname{matr}\left(\mathrm{R}_{\sigma_{2}}\left(\Phi_{\tilde{\epsilon}, \epsilon'}^{\mathbf{B}, \mathbf{I d}}\right)(y)\right)\right\|_{2}\\
				& \leq \frac{4}{3} \frac{1}{\lambda\beta } \frac{1}{\lambda \beta} \epsilon'  \le \frac{\epsilon}{2 \lambda}.
			\end{aligned}
		\end{equation*}
		
		For term II, by the triangle inequality, for every $y \in \mathcal{Y}$, we have the following expression:
		\begin{equation}\label{I}
			\begin{aligned}
				\left\|\operatorname{matr}\left(\mathrm{R}_{\sigma_{2}}\left(\Phi_{\tilde{\epsilon}, \epsilon^{\prime}}^{\textbf{B}, \textbf{Id}}\right)(y)\right)\right\|_{2} & \leq\left\|\operatorname{matr}\left(\mathrm{R}_{\sigma_{2}}\left(\Phi_{\tilde{\epsilon}, \epsilon^{\prime}}^{\textbf{B}, \textbf{Id}}\right)(y)\right)-\left(\textbf{Id}_{\mathbb{R}^{d(\bar{\epsilon})}}-\lambda \mathbf{B}_{y, \tilde{\epsilon}}^{\mathrm{rb}}\right)\right\|_{2}+\left\|\mathbf{I d}_{\mathbb{R}^{d(\bar{\epsilon})}}-\lambda \mathbf{B}_{y, \tilde{\epsilon}}^{\mathrm{rb}}\right\|_{2} \\
				& \leq \epsilon^{\prime}+1-\delta \leq 1-\delta+\frac{\lambda \beta}{4} \leq  1-\delta+\frac{\delta}{2}=1-\frac{\delta}{2} .
			\end{aligned}
		\end{equation}
		Thus, by Theorem \ref{3.4}, we derive that II $\leq \epsilon / 2 \lambda$. Combining I and II implies the following:
		\begin{equation*}
			\sup _{y \in \mathcal{Y}}\left\|\left(\mathbf{B}_{y, \tilde{\epsilon}}^{\mathrm{rb}}\right)^{-1}-\operatorname{matr}\left(\mathrm{R}_{\sigma_{2}}\left(\Phi_{\mathrm{inv} ; \tilde{\epsilon}, \epsilon}^{\textbf{B}}\right)(y)\right)\right\|_{2} \leq \epsilon.
		\end{equation*}
		For the size of neural network $\Phi_{\mathrm{inv} ; \tilde{\epsilon}, \epsilon}^{\textbf{B}}$, by Lemma \ref{11}, Theorem \ref{3.4}, and \eqref{I}, we have the following expression:
		\begin{equation*}
			L\left(\Phi_{\mathrm{inv} ; \tilde{\epsilon}, \epsilon}^{\textbf{B}}\right)=L\left(\Phi_{\mathrm{inv} ; \frac{\epsilon}{2 \lambda}}^{d(\tilde{\epsilon})} \odot \Phi_{\tilde{\epsilon}, \epsilon^{\prime}}^{\textbf{B}, \mathbf{I d}}\right)= 2l(\epsilon /(2 \lambda), \delta / 2) + 1 + B_{L}\left(\tilde{\epsilon}, \epsilon^{\prime}\right).
		\end{equation*} 
		and
		\begin{align*}
			M\left(\Phi_{\mathrm{inv} ; \tilde{\epsilon}, \epsilon}^{\textbf{B}}\right)&=M\left(\Phi_{\mathrm{inv} ; \frac{\epsilon}{2 \lambda}}^{d(\tilde{\epsilon})} \odot \Phi_{\tilde{\epsilon}, \epsilon^{\prime}}^{\textbf{B}, \textbf{Id}}\right)\\
			&\leq 5\Phi_{\mathrm{inv} ; \frac{\epsilon}{2 \lambda}}^{d(\tilde{\epsilon})} +5M(\Phi_{\tilde{\epsilon}, \epsilon^{\prime}}^{\textbf{B}, \textbf{Id}}) + 4d^2(\tilde{\epsilon})  \\
			&\leq5C_{\mathrm{inv}} d^3(\tilde{\epsilon}) l^2(\epsilon/2\lambda,\delta/2)+5B_{M}\left(\tilde{\epsilon}, \epsilon^{\prime}\right) + 9d^2(\tilde{\epsilon}).
		\end{align*}
		Therefore, by the definition of $l(\epsilon, \delta)$, $(ii)$ and $(iii)$ hold for suitably chosen constants $C^{B}_{L}=C^{B}_{L}(\alpha,\beta)$ and $C^{B}_{M}=C^{B}_{M}(\alpha,\beta)>0$. This completes the proof. 
	\end{proof}
	
	We can now construct ReQU neural networks that approximate the coefficient maps $\tilde{\mathbf{u}}^{\mathrm{h}}_{\cdot, \tilde{\epsilon}}, \mathbf{u}_{\cdot, \tilde{\epsilon}}^{\mathrm{rb}}.$ We present our main result.
	\begin{thm}\label{4.3}
		Let $\tilde{\epsilon} \geq \hat{\epsilon}$, $\epsilon >0,$ and define $\epsilon^{\prime}:= \frac{\epsilon}{\epsilon \beta + 2C_{f }},$\  $\epsilon^{\prime \prime}:= \frac{\epsilon\beta}{2} $, $\epsilon^{\prime \prime \prime}:= \min\{\frac{3 \epsilon' \lambda \beta^{2}}{8}, \frac{\lambda\beta}{4}\}$. Under Assumption \ref{4.1} and \ref{4.2},  there exist neural networks $ \Phi_{\tilde{\epsilon}, \epsilon}^{\mathbf{u}, \mathrm{rb}}$ and $\Phi_{\tilde{\epsilon}, \epsilon}^{\mathbf{u}, \mathrm{h}}$ satisfying the following properties:
		\begin{enumerate}[(i)]
			\item $\sup _{y \in \mathcal{Y}}\left|\mathbf{u}_{y, \tilde{\epsilon}}^{\mathrm{rb}}-\mathrm{R}_{\sigma_{2}}\left(\Phi_{\tilde{\epsilon}, \epsilon}^{\mathbf{u}, \mathrm{rb}}\right)(y)\right| \leq \epsilon$ and $\sup _{y \in \mathcal{Y}}\left|\tilde{\mathbf{u}}^{\mathrm{h}}_{y, \tilde{\epsilon}}-\mathrm{R}_{\sigma_{2}}\left(\Phi_{\tilde{\epsilon}, \epsilon}^{\mathbf{u}, \mathrm{h}}\right)(y)\right|_{\mathrm{G}} \leq \epsilon,$
			\item there exists a constant $C_{L}^{\mathrm{u}}=C_{L}^{\mathrm{u}}\left(\alpha,\beta,  C_{f }\right)>0$ such that
			$$
			\begin{aligned}
				L\left(\Phi_{\tilde{\epsilon}, \epsilon}^{\mathbf{u}, \mathrm{rb}}\right) \leq L\left(\Phi_{\tilde{\epsilon}, \epsilon}^{\mathbf{u}, \mathrm{h}}\right) 
				\leq C^{\mathrm{u}}_{L}\max\{ \log _{2}(\log_{2}(1/\epsilon))+B_{L}(\tilde{\epsilon},\epsilon^{\prime\prime\prime}),F_L(\tilde{\epsilon},\epsilon^{\prime\prime})\},
			\end{aligned}
			$$
			\item there exists a constant $C_{M}^{\mathrm{u}}=C_{M}^{\mathrm{u}}\left( \alpha,\beta, C_{\text {f }}\right)>0$ such that
			\begin{align*}
				M\left(\Phi_{\tilde{\epsilon}, \epsilon}^{\mathbf{u}, \mathrm{rb}}\right)  \leq& 5C^{\mathrm{u}}_Md^2(\tilde{\epsilon})(d(\tilde{\epsilon})\log^2_{2}(\log_{2}(1/\epsilon))+\log_{2}(\log_{2}(1/\epsilon))+B_{L}(\tilde{\epsilon},\epsilon^{\prime\prime\prime})+F_L(\tilde{\epsilon},\epsilon^{\prime\prime}))\\
				&+25B_M(\tilde{\epsilon},\epsilon^{\prime\prime\prime})+25F_{M}\left(\tilde{\epsilon}, \epsilon^{\prime \prime}\right) 
			\end{align*}
			\item $M\left(\Phi_{\tilde{\epsilon}, \epsilon}^{\mathbf{u}, \mathrm{h}}\right) \leq 5 D d(\tilde{\epsilon})+5 M\left(\Phi_{\tilde{\epsilon}, \epsilon}^{\mathrm{u}, \mathrm{rb}}\right).$
			
		\end{enumerate}
		
	\end{thm}

	\begin{proof}
		We define
		\begin{equation*}
			\Phi_{\tilde{\epsilon}, \epsilon}^{\mathbf{u}, \mathrm{rb}}:=\Phi_{\mathrm{mult}}^{d(\tilde{\epsilon}), d(\tilde{\epsilon}), 1} \odot \mathrm{P}\left(\Phi_{\mathrm{inv} ; \tilde{\epsilon}, \epsilon^{\prime}}^{\mathbf{B}}, \Phi_{\tilde{\epsilon}, \epsilon^{\prime \prime}}^{\mathbf{f}}\right) \bullet\left(\left(\left(\begin{array}{c}
				\mathbf{I} \mathbf{d}_{\mathbb{R}^{p}} \\
				\mathbf{I d}_{\mathbb{R}^{p}}
			\end{array}\right), \mathbf{0}_{\mathbb{R}^{2 p}}\right)\right) 
		\end{equation*}
		and
		\begin{equation*}
			\quad \Phi_{\tilde{\epsilon}, \epsilon}^{\mathbf{u},\mathrm{h}}:=\left(\left(\mathbf{V}_{\tilde{\epsilon}}, \mathbf{0}_{\mathbb{R}^{m}}\right)\right) \odot \Phi_{\tilde{\epsilon}, \epsilon}^{\mathbf{u}, \mathrm{rb}}.
		\end{equation*}
		
		According to the definition and triangle inequality, we obtain the following expression:
		\begin{equation*}
			\begin{aligned}
				&\left|\tilde{\mathbf{u}}^{\mathrm{h}}_{y, \tilde{\epsilon}}-\mathrm{R}_{\sigma_{2}}\left(\Phi_{\tilde{\epsilon}, \epsilon}^{\mathbf{u}, \mathrm{h}}\right)(y)\right|_{\mathrm{G}} 
				=\left|\mathbf{G}^{1 / 2} \cdot\left(\mathbf{V}_{\tilde{\epsilon}}\left(\mathbf{B}_{y, \tilde{\epsilon}}^{\mathrm{rb}}\right)^{-1} \mathbf{f}_{y, \tilde{\epsilon}}^{\mathrm{rb}}-\mathrm{R}_{\sigma_{2}}\left(\Phi_{\tilde{\epsilon}, \epsilon}^{\mathbf{u}, \mathrm{h}}\right)(y)\right)\right| \\
				&\leq\left|\mathbf{G}^{1 / 2} \mathbf{V}_{\tilde{\epsilon}} \cdot\left(\left(\mathbf{B}_{y, \tilde{\epsilon}}^{\mathrm{rb}}\right)^{-1} \mathbf{f}_{y, \tilde{\epsilon}}^{\mathrm{rb}}-\left(\mathbf{B}_{y, \tilde{\epsilon}}^{\mathrm{rb}}\right)^{-1} \mathrm{R}_{\sigma_{2}}\left(\Phi_{\tilde{\epsilon}, \epsilon^{\prime \prime}}^{\textbf{f}}\right)(y)\right)\right| \\
				&\quad +\left|\mathbf{G}^{1 / 2} \mathbf{V}_{\tilde{\epsilon}} \cdot\left(\left(\mathbf{B}_{y, \tilde{\epsilon}}^{\mathrm{rb}}\right)^{-1} \mathrm{R}_{\sigma_{2}}\left(\Phi_{\tilde{\epsilon}, \epsilon^{\prime \prime}}^{\textbf{f}}\right)(y)-\operatorname{matr}\left(\mathrm{R}_{\sigma_{2}}\left(\Phi_{\mathrm{inv} ; \tilde{\epsilon}, \epsilon^{\prime}}^{\textbf{B}}\right)(y)\right) \mathrm{R}_{\sigma_{2}}\left(\Phi_{\tilde{\epsilon}, \epsilon^{\prime \prime}}^{\textbf{f}}\right)(y)\right)\right| \\
				&\quad+\left|\mathbf{G}^{1 / 2} \cdot\left(\mathbf{V}_{\tilde{\epsilon}} \operatorname{matr}\left(\mathrm{R}_{\sigma_{2}}\left(\Phi_{\mathrm{inv} ; \tilde{\epsilon}, \epsilon^{\prime}}^{\textbf{B}}\right)(y)\right) \mathrm{R}_{\sigma_{2}}\left(\Phi_{\tilde{\epsilon}, \epsilon^{\prime \prime}}^{\textbf{f}}\right)(y)-\mathrm{R}_{\sigma_{2}}\left(\Phi_{\tilde{\epsilon}, \epsilon}^{\mathbf{u}, \mathrm{h}}\right)(y)\right)\right|\\
				&=: \mathrm{I}+\mathrm{II}+\mathrm{III} .
			\end{aligned}
		\end{equation*}
		
		For term $\mathrm{III}$, by Proposition \ref{p}, 
		\begin{align*}
			\mathrm{R}_{\sigma_{2}}\left(\Phi_{\tilde{\epsilon}, \epsilon}^{\mathbf{u},\mathrm{h}}\right)(y)&=\mathbf{V}_{\tilde{\epsilon}} \mathrm{R}_{\sigma_{2}}\left(\Phi_{\mathrm{mult}}^{d(\tilde{\epsilon}), d(\tilde{\epsilon}), 1} \odot \mathrm{P}\left(\Phi_{\mathrm{inv} ; \tilde{\epsilon}, \epsilon^{\prime}}^{\textbf{B}}, \Phi_{\tilde{\epsilon}, \epsilon^{\prime \prime}}^{\textbf{f}}\right)\right)(y, y)\\
			&=\mathbf{V}_{\tilde{\epsilon}} \operatorname{matr}\left(\mathrm{R}_{\sigma_{2}}\left(\Phi_{\mathrm{inv} ; \tilde{\epsilon}, \epsilon^{\prime}}^{\textbf{B}}\right)(y)\right) \mathrm{R}_{\sigma_{2}}\left(\Phi_{\tilde{\epsilon}, \epsilon^{\prime \prime}}^{\textbf{f}}\right)(y),
		\end{align*}
		which implies $\mathrm{III} = 0$. Therefore, it suffices to estimate $\mathrm{I}$ and $\mathrm{II}$. For term $\mathrm{I}$, by \eqref{B}, \eqref{G1/2}, Assumption \ref{4.2}, and the definition of $\epsilon^{\prime \prime}$, we have the following expression:
		\begin{equation*}
			\mathrm{I} \leq\left\|\mathbf{G}^{1 / 2} \mathbf{V}_{\tilde{\epsilon}}\right\|_{2}\left\|\left(\mathbf{B}_{y, \tilde{\epsilon}}^{\mathrm{rb}}\right)^{-1}\right\|_{2}\left|\mathbf{f}_{y, \tilde{\epsilon}}^{\mathrm{rb}}-\mathrm{R}_{\sigma_{2}}\left(\Phi_{\tilde{\epsilon}, \epsilon^{\prime \prime}}^{\textbf{f}}\right)(y)\right| \leq \frac{1}{\beta} \frac{\epsilon \beta}{2}=\frac{\epsilon}{2}.
		\end{equation*}
		For term II, note that 
		\begin{equation}\label{4.18}
			\sup _{y \in \mathcal{Y}}\left|\mathrm{R}_{\sigma_{2}}\left(\Phi_{\tilde{\epsilon}, \epsilon''}^{\textbf{f}}\right)(y)\right| \leq \epsilon''+C_{\mathrm{f}}.
		\end{equation} 
		By Proposition \ref{4.7} we obtain
		\begin{equation*}
			\begin{aligned}
				\mathrm{II} & \leq \left\|\mathbf{G}^{1 / 2} \mathbf{V}_{\tilde{\epsilon}} \cdot\left(\left(\mathbf{B}_{y, \tilde{\epsilon}}^{\mathrm{rb}}\right)^{-1}-\operatorname{matr}\left(\mathrm{R}_{\sigma_{2}}\left(\Phi_{\mathrm{inv} ; \tilde{\epsilon}, \epsilon^{\prime}}^{\mathrm{B}}\right)(y)\right)\right)\right\|_{2}\left|\mathrm{R}_{\sigma_{2}}\left(\Phi_{\tilde{\epsilon}, \epsilon^{\prime \prime}}^{\textbf{f}}\right)(y)\right| \leq \epsilon^{\prime} \cdot\left(\frac{\epsilon\beta}{2} + C_{\mathrm{f}}\right)= \frac{\epsilon}{2}.
			\end{aligned}
		\end{equation*}
		Combining the estimates on I, II, and III yields (i). Next, we estimate the size of the ReQU neural networks $\Phi_{\tilde{\epsilon}, \epsilon}^{\mathbf{u}, \mathrm{rb}}$. By Lemma \ref{11}, Lemma \ref{11b}, Proposition \ref{p}, and Proposition \ref{4.7}, we have the following expression:
		\begin{equation*}
			\begin{aligned}
				L\left(\Phi_{\tilde{\epsilon}, \epsilon}^{\mathbf{u}, \mathrm{rb}}\right) &<L\left(\Phi_{\tilde{\epsilon}, \epsilon}^{\mathbf{u},\mathrm{h}}\right) \leq 1+L\left(\Phi_{\mathrm{mult}}^{d(\tilde{\epsilon}), d(\tilde{\epsilon}), 1}\right)+L\left(\mathrm{P}\left(\Phi_{\mathrm{inv} ; \tilde{\epsilon}, \epsilon^{\prime}}^{\textbf{B}}, \Phi_{\tilde{\epsilon}, \epsilon^{\prime \prime}}^{\textbf{f}}\right)\right) \\
				& \leq 1+2+\max \left\{L\left(\Phi_{\mathrm{inv} ; \tilde{\epsilon}, \epsilon^{\prime}}^{\textbf{B}}\right), F_{L}\left(\tilde{\epsilon}, \epsilon^{\prime \prime}\right)\right\} \\
				& \leq C^{\mathrm{u}}_{L}\max\{ \log _{2}(\log_{2}(1/\epsilon))+B_{L}(\tilde{\epsilon},\epsilon^{\prime\prime\prime}),F_L(\tilde{\epsilon},\epsilon^{\prime\prime})\},
			\end{aligned}
		\end{equation*}
		where $C_{L}^{\mathrm{u}}$ is a suitable constant $C_{L}^{\mathrm{u}}=C_{L}^{\mathrm{u}}\left( \alpha, \beta, C_f\right)>0.$
		Furthermore, we have the following expression:
		\begin{equation}\label{5.4}
			\begin{aligned}
				M\left(\Phi_{\tilde{\epsilon}, \epsilon}^{\mathbf{u}, \mathrm{rb}}\right) &\leq 5 M\left(\Phi_{\mathrm {mult }}^{d(\tilde{\epsilon}), d(\tilde{\epsilon}), 1}\right)+5 M\left(\mathrm{P}\left(\Phi_{\mathrm{inv} ; \tilde{\epsilon}, \epsilon^{\prime}}^{\textbf{B}}, \Phi_{\tilde{\epsilon}, \epsilon^{\prime \prime}}^{\textbf{f}}\right)\right) + 4d^2(\tilde{\epsilon})\\
				&\leq 64d^2(\tilde{\epsilon})+5 M\left(\mathrm{P}\left(\Phi_{\mathrm{inv} ; \tilde{\epsilon}, \epsilon^{\prime}}^{\textbf{B}}, \Phi_{\tilde{\epsilon}, \epsilon^{\prime \prime}}^{\textbf{f}}\right)\right).
			\end{aligned}
		\end{equation}
		The second term of the above equation can be estimated as follows:
		\begin{equation}\label{5.5}
			\begin{aligned}
				&M\left(\mathrm{P}\left(\Phi_{\mathrm{inv} ; \tilde{\epsilon}, \epsilon^{\prime}}^{\mathrm{B}}, \Phi_{\tilde{\epsilon}, \epsilon^{\prime \prime}}^{\mathrm{f}}\right)\right) \\
				&\leq 5M\left(\Phi_{\mathrm{inv} ; \tilde{\epsilon}, \epsilon^{\prime}}^{\mathrm{B}}\right)+5M\left(\Phi_{\tilde{\epsilon}, \epsilon^{\prime \prime}}^{\mathrm{f}}\right)+2d^2(\tilde{\epsilon})(20\max \left\{L\left(\Phi_{\mathrm{inv}; \tilde{\epsilon}, \epsilon^{\prime}}^{\mathrm{B}}\right), F_{L}\left(\tilde{\epsilon}, \epsilon^{\prime \prime}\right)\right\}+8)\\
				&\leq 5C^{B}_{M}d^3(\tilde{\epsilon})\log^{2}_{2}(\log_{2}(1/\epsilon^{\prime}))+5B_M(\tilde{\epsilon},\epsilon^{\prime\prime\prime})+5F_{M}\left(\tilde{\epsilon}, \epsilon^{\prime \prime}\right)\\
				&+2d^2(\tilde{\epsilon})(20C^{\mathrm{u}}_{L}\max\{ \log _{2}(\log_{2}(1/\epsilon^{\prime}))+B_{L}(\tilde{\epsilon},\epsilon^{\prime\prime\prime}),F_L(\tilde{\epsilon},\epsilon^{\prime\prime})\}+8)\\
				&\leq C^{\mathrm{u}}_Md^2(\tilde{\epsilon})(d(\tilde{\epsilon})\log^2_{2}(\log_{2}(1/\epsilon))+\log_{2}(\log_{2}(1/\epsilon))+B_{L}(\tilde{\epsilon},\epsilon^{\prime\prime\prime})+F_L(\tilde{\epsilon},\epsilon^{\prime\prime}))+5B_M(\tilde{\epsilon},\epsilon^{\prime\prime\prime})+5F_{M}\left(\tilde{\epsilon}, \epsilon^{\prime \prime}\right)
			\end{aligned}
		\end{equation}
		for a suitably chosen constant $C_{M}^{\mathrm{u}}=C_{M}^{\mathrm{u}}\left(C_{B}, C_{L}^{\mathrm{u}}\right)=C_{L}^{\mathrm{u}}\left(\alpha, \beta, C_f\right)>0 .$ Combining \eqref{5.4} and \eqref{5.5} yields (iii), then (iv) follows immediately by Lemma \ref{11} and the definition of $\Phi_{\tilde{\epsilon}, \epsilon}^{\mathbf{u}, \mathrm{h}}$. The proof is completed.
	\end{proof}
			
	\section{Numerical Experiment}
	In this section, we present numerical results to verify our theoretical analysis. The parameter-dependent PDE is considered to be parametric diffusion equations with homogeneous Dirichlet boundary conditions
	\begin{equation*}
		-\nabla \cdot\left(a_{y}(\mathbf{x}) \cdot \nabla u_{y}(\mathbf{x})\right)=f_y(\mathbf{x}), \quad \text { on } \Omega=(0,1)^{2},\left.\quad u_{y}\right|_{\partial \Omega}=0
	\end{equation*}
	where $f_y \in L^{2}(\Omega)$ and $a_y \subset L^{\infty}(\Omega)$.
	
	\subsection{Setup of Neural Networks}
	The experiment is implemented using PyTorch, \cite{paszke2019pytorch}. We use fully connected neural networks with architecture $$S=(\mathrm{dim}_{\mathrm{in}}, 300, \ldots, 300, \mathrm{dim}_{\mathrm{out}})$$ of different numbers of layers $L=3, 5, 7, 9, 11$ respectively, where the weights and biases are initialized according to Xavier initialization. We use 20000 training examples, 5000 validation examples, and 5000 test examples, drawn with respect to the uniform probability measure on $\mathcal{Y}$. The optimization is performed through batch gradient descent, and the batch size is 256. We use the ADAM optimizer with fixed hyper-parameters: $lr = 2.0\times 10^{-4}, \beta_{1} = 0.9, \beta_{2} = 0.999 $, and $\epsilon = 1.0\times 10^{-8}$. The training process is stopped after reaching 40,000 epochs. Furthermore, the loss function is the relative error on the high-fidelity discretization of $H$
	$$
	\mathcal{L}: \mathbb{R}^{m} \times\left(\mathbb{R}^{m} \backslash\{0\}\right) \rightarrow \mathbb{R}, \quad \left(\mathrm{x}_{1}, \mathrm{x}_{2}\right) \mapsto \frac{\left|\mathrm{x}_{1}-\mathrm{x}_{2}\right|_{\mathbf{G}}}{\left|\mathrm{x}_{2}\right|_{\mathbf{G}}}.
	$$
	\begin{rem}
		In practice, we use the discrete version of the mean relative error with respect to the parameter set $\mathcal{Y}$ 
		$$\int_{\mathcal{Y}} \frac{\left|\tilde{\mathbf{u}}^{\mathrm{h}}_{y, \tilde{\epsilon}}-\mathrm{R}_{\sigma_{2}}\left(\Phi_{\tilde{\epsilon}, \epsilon}^{\mathbf{u}, \mathrm{h}}\right)(y)\right|_{\mathbf{G}}}{|\tilde{\mathbf{u}}^{\mathrm{h}}_{y, \tilde{\epsilon}}|_{\mathbf{G}}} dy  $$
		in our numerical experiments, instead of the uniform approximation error in the theoretical analysis. These two errors are comparable, and thus, imply the same rates. 
	\end{rem}
	
	\subsection{Parametric  Sets}\label{section6.1}
	We consider two types of parametric diffusion PDEs as follows.
	\begin{description}
		\item[Parametric Diffusion Coefficients] To compare with the ReLU neural network in \cite{geist2021numerical}, we set 
		$$f(\mathbf{x}) = 20 + 10x_1 -5 x_2,$$
		for $\mathbf{x} = (x_1,x_2)\in \Omega= [0,1]^2$, which is independent of the parameters. Next, we parametrize the diffusion coefficient set $\{a_y: y\in \mathcal{Y}\}\subset \mathbb{R}^{p}$ for $p=s^2 \in \mathbb{N}$ as follows:
		\begin{equation*}
			\mathcal{A}(p, \mu):=\left\{\mu+\sum_{i=1}^{p} y_{i} \mathcal{X}_{\Omega_{i}}: y \in \mathcal{Y}=[0,1]^{p}\right\},
		\end{equation*}
		where $\left(\Omega_{i}\right)_{i=1}^{p}$ forms a $s \times s$ chessboard partition of $(0,1)^{2}$, and $\mu>0$ is a fixed shift. In our numerical tests, we select shifts $\mu=10^{-1}$, $p=s^{2}$, and $s=3$, which yield $p=9$.
		
		The dataset used in this case comes from \cite{geist2021numerical}; it is available at \url{www.github.com/MoGeist/diffusion_PPDE}. The dataset is produced by FEniCS \cite{alnaes2015fenics}, which is based on the finite element method. The finite element space $H^h$ is constructed by the triangulation of $\Omega=[0,1]^{2}$ with $101 \times 101=10201$ equidistant grid points and first-order Lagrange finite elements. This space shall serve as a  discretized version of the space $H^{1}(\Omega) .$ In this case, the dimension of the high-fidelity space $H^h$ is 10201, the input dimension is $\mathrm{dim}_{\mathrm{in}}=9$, and the output dimension is $\mathrm{dim}_{\mathrm{out}}=10201$. An efficient performance can be obtained if the numerical experiments are conducted in the reduced basis space, which is beyond the scope of this work.
		
		\item[Parametric Forcing Term] We consider the Laplace equation with the fixed diffusion coefficient 
		$a(\mathbf{x})= 1$
		for $\mathbf{x} = (x_1,x_2)\in \Omega=[0,\pi]^2$, which is independent of the parameters. We parametrize the forcing term  $\{f_y(\mathbf{x}): y\in \mathcal{Y}\}$ as follows:
		$$\mathcal{F} := \{\sin y_1 \sin x_1 \sin x_2 + \sin y_2 \sin x_1 \sin 2x_2 +\sin y_3\sin 2x_1 \sin x_2 +\sin y_4 \sin 2x_1 \sin 2x_2,\},  $$
		where $y=(y_1,y_2,y_3,y_4)\in \mathcal{Y}= [0,\pi]^4$. The Laplace equation could be solved explicitly for every parameter $y$, drawn with respect to the uniform probability measure on $\mathcal{Y}$, which forms the dataset in our test. In this case, the dimension of high-fidelity space $H^h$ is four. The input dimension $\mathrm{dim}_{\mathrm{in}}$ and output dimension  $\mathrm{dim}_{\mathrm{out}}$ are both four.
	\end{description}

	\subsection{Numerical Results}
	In this subsection, we present our numerical results to verify that ReQU neural networks are numerically stable and more efficient for approximating the discretized parameter to solution map, compared to  ReLU neural networks.
	
	\paragraph{Parametric Diffusion Coefficients}
	We present the following mean relative errors of approximating the discretized parameter to solution map using the ReQU and ReLU neural networks. To ensure comparability of the networks, we use the same structure (fully connected neural network with 11 layers), and train them with the ADAM optimizer with the same parameter as prescribed previously.  Hereafter, if not otherwise stated, we always perform five times random Xavier uniform initializations of the ReLU and ReQU neural networks with fixed partition of datasets and calculate the mean and standard deviation of the test errors of them.
	
\begin{table}[H]
	\centering
	\begin{tabular}{cc}
		\toprule  
		&Mean relative test error \\ 
		\midrule  
		ReQU neural network&0.003255 $\pm$ 0.000148 \\
		\midrule  
		ReLU neural network&0.007126 $\pm$ 0.000181 \\
		\bottomrule  
		\end{tabular}
		\caption{Test errors of the ReQU and ReLU neural networks with 11 layers.}
\end{table}
	With the same network structure, the error of the ReQU neural network is reduced by more than 50$\%$ compared with the ReLU neural network. Figure \ref{fig:1}  illustrates a comparison of the training loss curves of the ReLU and ReQU neural networks, where the ReLU neural network has a larger error and slower convergence speed.
	\begin{figure}[H]
		\centering
		\includegraphics[width=0.6\textwidth]{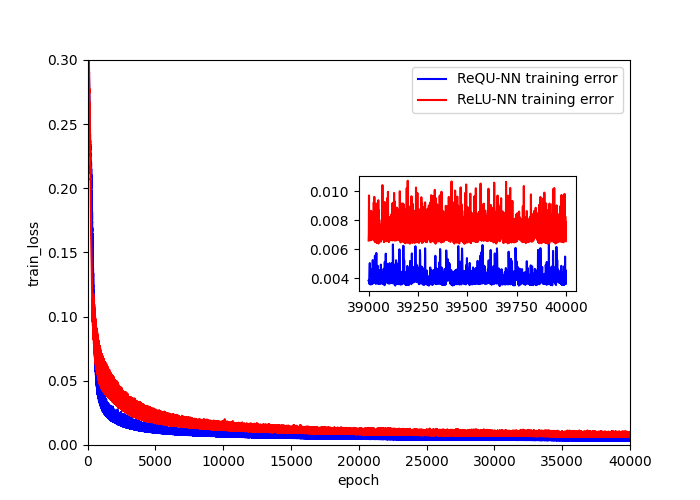}
		\caption{Training loss curve of approximating the discretized parameter to solution map using the ReLU and ReQU neural networks  with 11 layers, both trained by the ADAM optimizer.}
		\label{fig:1}
	\end{figure} 
	In Figure \ref{fig:2}, we depict the ground truth solution and an average performing solution predicted by the ReQU neural network, which displays a visualization highlighting the accuracy of the ReQU neural network.
	\begin{figure}[H]
		\centering 
		\includegraphics[width=0.7\textwidth]{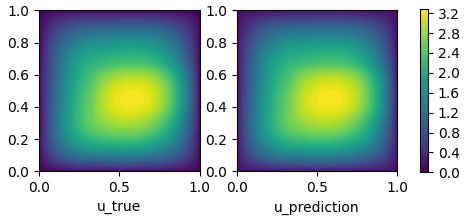}
		\caption{Ground truth solution and the prediction of  ReQU neural network with 11 layers. }
		\label{fig:2}
	\end{figure}
	
	Next, we present some results of approximating the discretized parameter to solution map with fully connected ReQU neural networks with architecture $S=(9, 300, \ldots, 300,10201)$ of different numbers of layers $L=3, 5, 7, 9$.\\
	\begin{table}[H]
		\centering
		\begin{tabular}{ccc}
			\toprule  
			Layers & ReQU neural network & ReLU neural network \\ 
			\midrule  
			$3$&$0.003860 \pm 0.000196$&$ 0.005789 \pm 0.000240$  \\
			\midrule  
			5&0.003455 $\pm$ 0.000054& 0.006308 $\pm$ 0.000229\\
			\midrule  
			7&0.003434 $\pm$ 0.000156& 0.006708 $\pm$ 0.000194\\
			\midrule  
			9&0.003414 $\pm$ 0.000318& 0.007023 $\pm$ 0.000087\\
			\bottomrule  
		\end{tabular}
		\caption{Test errors of the ReQU and ReLU neural networks with different numbers of layers.}
		\label{tab:table2}
	\end{table}
 According to Table 2, ReQU neural networks can achieve smaller errors than ReLU neural networks, at identical network architecture hyperparameters (layers L = 3, 5, 7, 9).
Moreover,  the test errors of ReQU neural networks are more stable as the numbers of layers increase, thus ReQU neural networks are more robust with respect to overfitting. In Figure \ref{Fig:3},  we plot the training loss curves of the ReLU and ReQU neural networks with different numbers of layers, which verify the faster convergence speeds of ReQU neural networks.
    \begin{figure}[H]
    	\centering
    	\begin{minipage}{0.49\linewidth}
    		\centering
    		\includegraphics[width=0.9\linewidth]{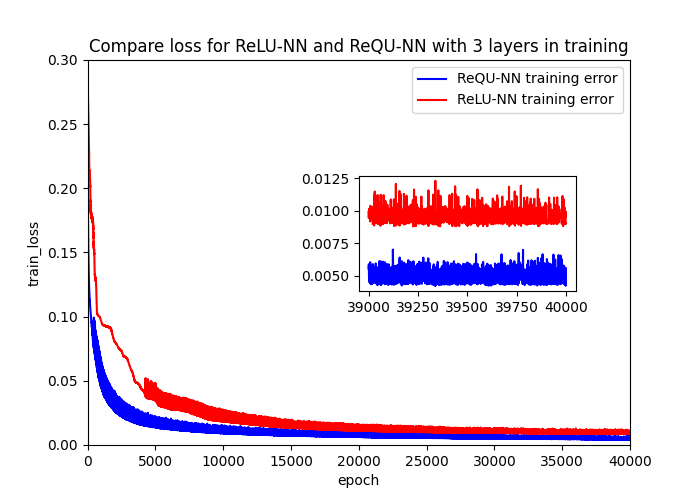}
    	\end{minipage}
    	\begin{minipage}{0.49\linewidth}
    		\centering
    		\includegraphics[width=0.9\linewidth]{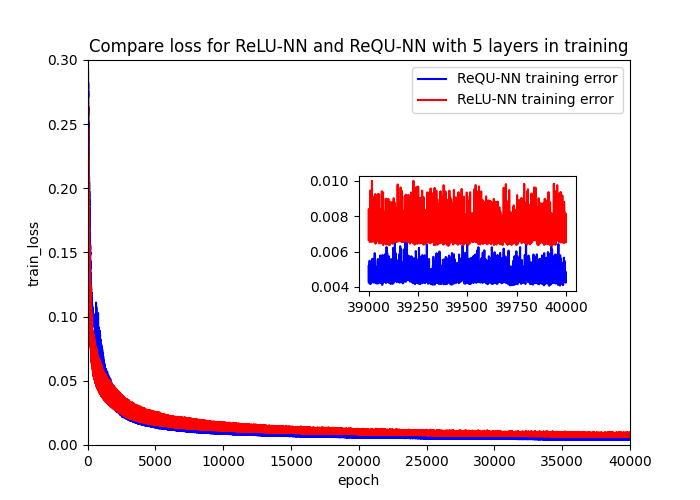}
    	\end{minipage}
    	\begin{minipage}{0.49\linewidth}
    		\centering
    		\includegraphics[width=0.9\linewidth]{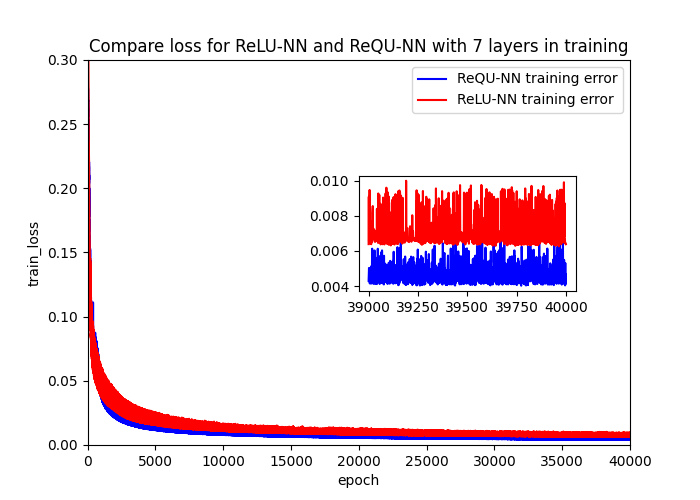}
    	\end{minipage}
    	\begin{minipage}{0.49\linewidth}
    		\centering
    		\includegraphics[width=0.9\linewidth]{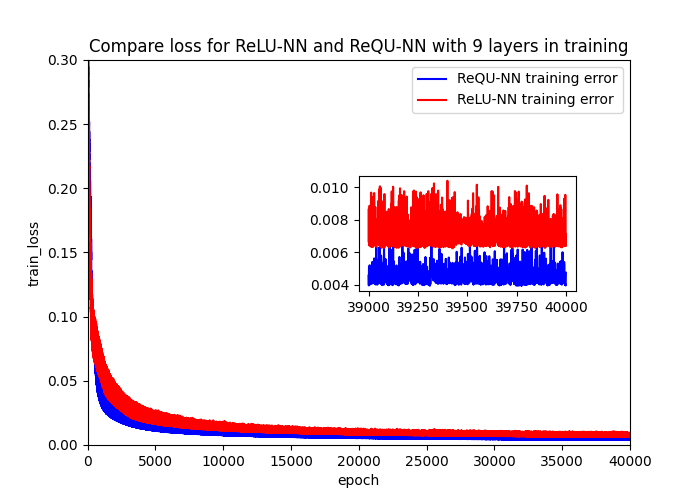}
    	\end{minipage}
    \caption{Training loss curve of  ReLU and ReQU neural networks  with different numbers of layers.}
    \label{Fig:3}
    \end{figure}

	We then use the trained neural networks, for every test sample, to calculate the time required for computing the solution through a simple forward pass. And we present the average and  standard deviation of the results in Table \ref{time}. 
	\begin{table}[H]
		\centering
		\begin{tabular}{ccc}
			\toprule  
			Layers & ReQU neural network & ReLU neural network \\ 
			\midrule  
			3&0.000575 $\pm$ 0.000027& 0.000532 $ \pm$ 0.000005\\
			\midrule  
			5&0.000578 $\pm$ 0.000025& 0.000610 $\pm$ 0.000009\\
			\midrule  
			7&0.000579 $\pm$ 0.000022& 0.000611 $\pm$ 0.000023\\
			\midrule  
			9&0.000582 $\pm$ 0.000022& 0.000619 $\pm$ 0.000031\\
			\midrule  
			11&0.000596 $\pm$ 0.0000017& 0.000629 $\pm$ 0.000025\\
			\bottomrule  
		\end{tabular}
		\caption{ Inference time (in seconds) of ReQU and ReLU neural networks with different numbers of layers.}
		\label{time}
	\end{table}
    Combine the results  from Table \ref{tab:table2} and Table  \ref{time}, we draw the inference time versus error curve of ReLU and ReQU neural networks.
    
    	\begin{figure}[H]
    	\centering
    	\includegraphics[width=0.6\textwidth]{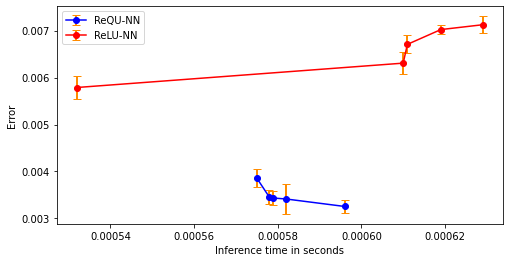}
        \caption{Inference time versus error curve of ReLU and ReQU neural network.}
    \end{figure}

   From above, we can conclude that ReQU neural networks can achieve better accuracy with fewer inference time than ReLU neural networks when the number of layers is 5, 7, 9 and 11. However, the inference time for ReQU neural networks can be larger than that of ReLU neural networks when the number of layers equals 3, while ReQU neural networks still achieve better accuracy. 
	
	\paragraph{Parametric Forcing Term}
	We present the mean relative errors of approximating the discretized parameter to solution map by the fully connected ReQU and ReLU neural networks with the same architecture $S=(4, 300, \ldots, 300,4)$ of $L=3, 5, 7, 9, 11$ layers, respectively.
	
	\begin{table}[H]
		\centering
		\begin{tabular}{ccc}
			\toprule  
			Layers & ReQU neural network& ReLU neural network \\ 
			\midrule  
			3&0.000450 $\pm$ 0.000020&0.000654 $\pm$ 0.000041 \\
			\midrule  
			5&0.000629 $\pm$ 0.000021&0.001276 $\pm$ 0.000068 \\
			\midrule  
			7&0.000682 $\pm$ 0.000123&0.001433 $\pm$ 0.000078 \\
			\midrule  
			9&0.000773 $\pm$ 0.000016& 0.001602 $\pm$ 0.000073\\
			\midrule  
			11&0.000844 $\pm$ 0.000051&0.001747 $\pm$ 0.000142\\
			\bottomrule  
		\end{tabular}
		\caption{Test errors of the ReQU and ReLU neural networks with different numbers of layers.}
		\label{tab:table4}
	\end{table}
	The results from Table \ref{tab:table4} show ReQU neural networks  achieve smaller test errors and are more robust with respect to overfitting.  In Figure \ref{fig:3}, we depict the ground truth solution and an average performing solution predicted by the ReQU neural network with 3 layers.
	\begin{figure}[H]
		\centering 
		\includegraphics[width=0.7\textwidth]{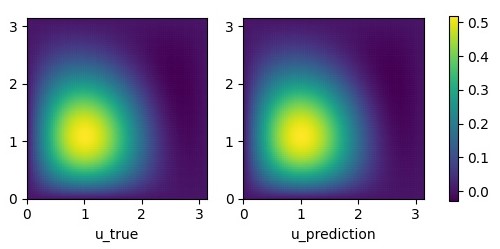}
		\caption{Ground truth solution and the prediction of  ReQU neural network with 3 layers. }
		\label{fig:3}
	\end{figure}

     \section{Discussion}
   In this section, we will discuss our results in terms of the dependence on the dimension of the reduced basis  $d(\tilde{\epsilon})$ and the suitability of the two assumptions \ref{4.1} and \ref{4.2}.  
   \subsection{Dependence on the Dimension of the Reduced Basis  $d(\tilde{\epsilon})$}
   To approximate parametric map 
   $$\mathcal{Y} \rightarrow \mathbb{R}^{d(\tilde{\epsilon})}: \quad y \mapsto \mathbf{u}_{y, \tilde{\epsilon}}^{\mathrm{rb}}$$
   with $\epsilon$ accuracy, by Theorem \ref{4.3} the number of non-zero weights of ReQU neural network required  is $\mathcal{O}(d^3(\tilde{\epsilon})\log^2_{2}(\log_{2}(1/\epsilon)),$ while the corresponding  known results is $\mathcal{O}(d(\tilde{\epsilon})^{3}  \log_{2}(1 / \epsilon))$ for $\operatorname{ReLU}$ network.  This implies a significant advantage of ReQU over ReLU in approximating the parametric map when $d(\tilde{\epsilon})\ll \log_{2}(\log_{2}(1/\epsilon)).$ We remark that low rank approximations by reduced basis is achievable under certain circumstance, especially when the parametric set $\mathcal{Y}$ is compact. For example, the authors in \cite{bachmayr2017kolmogorov} gave explicit low-rank representations when the diffusion coefficients are piecewise constant over a partition of the physical domain.
   
   We illustrate this by discussing our first numerical experiment on parametric diffusion coefficients and showing that $d(\Tilde{\epsilon})$ is relatively small in this case. Recall that $f(\mathbf{x}) = 20 + 10x_{1} -5 x_{2},$ for $\mathbf{x} = (x_1,x_2)\in \Omega= [0,1]^2$, which is independent of the parameters. And we parametrize the diffusion coefficient set $\{a_y: y\in \mathcal{Y}\}\subset \mathbb{R}^{p}$ for $p=s^2 \in \mathbb{N}$ as 
   \begin{equation*}
   	\mathcal{A}(p, \mu):=\left\{\mu+\sum_{i=1}^{p} y_{i} \mathcal{X}_{\Omega_{i}}: y \in \mathcal{Y}=[0,1]^{p}\right\},
   \end{equation*}
   where $\left(\Omega_{i}\right)_{i=1}^{p}$ forms a $s \times s$ chessboard partition of $(0,1)^{2}$, and $\mu>0$ is a fixed shift. In our numerical tests, we select shifts $\mu=10^{-1}$, $p=s^{2}$, and $s=3$, which yield $p=9$. In the interior of each $\Omega_{i}$, we have $$
   -\left(\mu+\sum_{i=1}^{p} y_{i} \mathcal{X}_{\Omega_{i}}\right) \Delta u_{y}(x)=20+10 x_{1}-5 x_{2}.
   $$
   Then, we can exactly solve $u_y$ as
   \begin{eqnarray*}
   	u_{y}=  -\frac{1}{\mu +y_{i}}&\Bigg[& 20\left(c_{1} \frac{x_{1}^{2}}{2}+d_{1} \frac{x_{2}^{2}}{2}\right)+10\left(c_{2} \frac{x_{1}^{3}}{6}+d_{2} \frac{x_{1} x_{2}^{2}}{2}\right)+5\left(c_{3} \frac{x_{1}^{2}}{2} x_{2}+d_{3} \frac{x_{2}^{3}}{6}\right)\\
   	&& +  c_4 x_1 x_2+c_5 x_1+c_6 x_2 + c_7\Bigg]
   \end{eqnarray*}
   where $c_i + d_i = 1$, for $i\in{1,2,3}.$ Thus, we obtain that
   $$
   \left.u_y\right|_{\Omega_{i}} \in \operatorname{span}\left\{x_{1}^{2}\mathcal{X}_{\Omega_{i}}, x_{2}^{2}\mathcal{X}_{\Omega_{i}}, x_{1}^{3}\mathcal{X}_{\Omega_{i}},  x_{1} x_{2}^{2} \mathcal{X}_{\Omega_{i}}, x_{1}^{2} x_{2} \mathcal{X}_{\Omega_{i}}, x_{2}^{3} \mathcal{X}_{\Omega_{i}}, x_1 x_2 \mathcal{X}_{\Omega_{i}}, x_1 \mathcal{X}_{\Omega_{i}}, x_2 \mathcal{X}_{\Omega_{i}}, \mathcal{X}_{\Omega_{i}} \right\}.
   $$
   This yields an upper bound of $10\times 9 = 90$ degrees of freedom for $u_y$. Since there are $24$ continuity conditions (independent of $y$ ) on $u$ at domain and subinterval boundaries, we have
   $$
   \operatorname{rank}(u) \leq 90 - 24 = 66
   $$
   and therefore $d(\tilde{\epsilon})\le  66$, independent of the particular choice of $\tilde{\epsilon}$.

   \subsection{Suitability of Assumptions \ref{4.1} and \ref{4.2}}
   Indeed, the two assumptions  \ref{4.1} and \ref{4.2} hold if the maps $y \mapsto b_{y}(u, v)$ and $y \rightarrow f_{y}(v)$ are continuous for all $u, v \in \mathcal{H}$. 
   In this case, it is enough to show that ReQU neural networks can approximate continuous function. By \cite[Theorem3.1]{li2019better}, ReQU neural network can represent multivariate polynomial on $\mathbb{R}^{d}$ with no error. Thus, by the Weierstrass approximation theorem, assumption 5.1 and 5.2 can be achieved. 
   
   Next, we discuss about the required size for the ReQU-realizations of the maps $y \mapsto \mathbf{B}_{y, \tilde{\epsilon}}^{\mathrm{rb}}$ and $y \mapsto \mathbf{f}_{y, \tilde{\epsilon}}^{\mathrm{rb}}$, taking  our first numerical experiment for example again. We write the parametric diffusion equation 
   $$
   -\nabla \cdot\left(a_{y}(\mathbf{x}) \cdot \nabla u_{y}(\mathbf{x})\right)=f(\mathbf{x}), \quad \text { on } \Omega=(0,1)^{2},\left.\quad u_{y}\right|_{\partial \Omega}=0
   $$
   in its variation form
   $$
   B_{y}\left(u_{y}, v\right):=\int_{\Omega} \mu \nabla u_{y} \nabla v \mathrm{~d} \mathbf{x}+\sum_{i=1}^{p} y_{i} \int_{\Omega} \mathcal{X}_{\Omega_{i}} \nabla u_{y} \nabla v \mathrm{~d} \mathbf{x}=f(v), \quad \text { for all } v \in \mathcal{H}.
   $$
   By the definition of 	$\mathbf{B}_{y, \tilde{\epsilon}}^{\mathrm{rb}}:=\left(B_{y}\left(\psi_{j}, \psi_{i}\right)\right)_{i, j=1}^{d(\tilde{\epsilon})}$ and Lemma 2.2, we can derive
   \begin{equation*}
   	y_i =\mathbf{\beta}_{1}^{T} \sigma_{2}\left(\mathbf{\omega}_{1} y_i+\mathbf{\gamma}_{1}\right), 
   \end{equation*}
   where $\mathbf{\omega}_{1}=[1,-1,1,-1]^{T},\mathbf{\gamma}_{1}=[1,-1,-1,1]^{T}$ and $\mathbf{\beta}_{1}=\frac{1}{4}[1,1,-1,-1]^{T}$. Thus, we have for $\tilde{\epsilon}, \epsilon>0$,
   $$
   \begin{aligned}
   	&B_{L}(\tilde{\epsilon}, \epsilon)=2,\\
   	&F_{L}(\tilde{\epsilon}, \epsilon)=1,\\ 
   	& B_{M}(\tilde{\epsilon}, \epsilon) \leq \left(\left\|\omega_{1} \right\|_{0}+\left\|\gamma_{1}\right\|_{0}\right) p+(4p+1) d(\tilde{\varepsilon})^{2}=8p + (4p+1) d(\tilde{\varepsilon})^{2} , \\
   	& F_{M}(\tilde{\epsilon}, \epsilon) \leq p d(\tilde{\epsilon}).
   \end{aligned}
   $$
   Combining this observation with the
   statement of Theorem \ref{4.3}, we can conclude that the governing quantity in the obtained number of layers and complexity bound
   are given by $\mathcal{O}(\log^2_{2}(\log_{2}(1/\epsilon))$ and $\mathcal{O}(d^3(\tilde{\epsilon})\log^2_{2}(\log_{2}(1/\epsilon))$ respecitively.

	\section{Conclusion}
	
	In this paper, we give constructive proofs of complexity bounds for approximating parametric map by deep neural networks using ReQU function as the activation functions. In contrast to the previously established complexity-bound $\mathcal{O}\left(d^3\log_{2}^{q}(1/ \epsilon) \right)$ for ReLU neural networks, we derive an upper bound $\mathcal{O}\left(d^3\log_{2}^{q}\log_{2}(1/ \epsilon) \right)$ on the size of the deep ReQU neural network required to achieve accuracy $\epsilon>0$. This implies the better performance of deep ReQU neural networks in solving parametric partial differential equations.  In addition, we compare the practical performance of deep neural networks approximation based on ReLU versus ReQU on two parametric PDE models. These numerical experiments indicate that ReQU neural networks can achieve smaller errors than ReLU neural networks, at identical network architecture hyperparameters (depth and layer widths).  Furthermore, ReQU neural networks converge more quickly in the optimization and are more stable with respect to overfitting.
	
	 This work opens up many interesting directions  for future research. First of all, for both ReLU and ReQU based methods, the reduced basis approximation for model reduction on function spaces could be improved by adaptive data-driven approaches such as the kernel proper orthogonal decomposition (KPOD) and principal component analysis (PCA) techniques, see \cite{bhattacharya2021model,salvador2021non}. Furthermore, a limitation of this study is that our findings indicate ReQU neural networks  with sufficient layers and size can provide accurate approximations, but we do not know how to choose the ideal neural network architectures for the specific problem in applications. As a future work, we plan to conduct a comprehensive experiment  to investigate the theoretical setup of this paper, including using other special architectures like convolutional neural networks and generalizing to more general parametric problems. Finally, we  also want to apply our methodology  to more difficult PDE problems, such as complex PDE models in the sciences and in engineering and numerical solution of forward and inverse problems of parametric and stochastic PDEs.
    
     \paragraph{Acknowledgement}
	We would like to thank the two anonymous reviewers for their helpful and constructive comments on our work.
	\paragraph{Funding}
	The authors were in part supported by NSFC (Grant No. 11725102), National Support Program for Young Top-Notch Talents, and Shanghai Science and Technology Program (Project No. 21JC1400600, No. 19JC1420101, and No. 20JC1412700). 
	
  \paragraph{Data Availability}	
   The dataset  is available at \url{www.github.com/MoGeist/diffusion_PPDE}.
	
	\section*{Declarations}	
	 \paragraph{Competing interests} The authors have not disclosed any competing interests.

	
	\frenchspacing
	\bibliographystyle{plain}
	\bibliography{ref}
	
\end{document}